\documentclass[smallextended,numbook,runningheads]{svjour3}     
\RequirePackage{fix-cm}

\smartqed  
\usepackage{graphicx}
\usepackage{amsfonts}
\usepackage{stmaryrd}
\usepackage{amssymb}
\usepackage{amsbsy}
\usepackage{dsfont}
\usepackage{mathtools}
\usepackage[misc]{ifsym}
 \journalname{}

\usepackage{amsmath,amsfonts,amssymb}
\usepackage{xcolor,mhsetup}
\usepackage[extra,safe]{tipa}
\usepackage{mathtools}
\usepackage[applemac]{inputenc}
\usepackage{mathrsfs}
\usepackage{mathabx}

\usepackage{amsthm}
\usepackage{mathptmx}

\newtheorem{Def}{Definition}[section]

\newtheorem{Lem}[Def]{Lemma}
\newtheorem{Thm}[Def]{Theorem}
\newtheorem{Cor}[Def]{Corollary}

\theoremstyle{definition}
\newtheorem{Rem}[Def]{Remark}

\newcommand{\e}{\mathbb{E}}
\newcommand{\real}{\mathbb{R}}
\newcommand{\n}{\mathbb{N}}

\newcommand{\1}{{\bf 1}}

\newcommand{\F}{\mathcal{F}}

\newcommand{\ff}{\mathbb{F}}

\allowdisplaybreaks

\begin{document}
		
\title{On a positivity preserving numerical scheme for jump-extended CIR process: the alpha-stable case.}

\titlerunning{On a positivity preserving scheme for the alpha-CIR process}

\author{Libo Li \and
	    {Dai Taguchi}}

\authorrunning{L. Li and D. Taguchi} 

\institute{Libo Li (\Letter) \at
           School of Mathematics and Statistics, University of New South Wales, NSW, Australia.\\
           \email{libo.li@unsw.edu.au}           
			\and
			Dai Taguchi \at
		Graduate School of Engineering Science, Osaka University, 1-3, Machikaneyama-cho, Toyonaka, Osaka, Japan, \\
		\email{dai.taguchi.dai@gmail.com}
}

\date{Received: date / Accepted: date}

\date{}
\maketitle
\begin{abstract}
We propose a positivity preserving implicit Euler-Maruyama scheme for a jump-extended Cox-Ingersoll-Ross (CIR) process where the jumps are governed by a compensated spectrally positive $\alpha$-stable process for $\alpha \in (1,2)$. Different to the existing positivity preserving numerical schemes for jump-extended CIR or CEV (Constant Elasticity Variance) process, the model considered here has infinite activity jumps. We calculate, in this specific model, the strong rate of convergence and give some numerical illustrations. Jump extended models of this type were initially studied in the context of branching processes and was recently introduced to the financial mathematics literature to model sovereign interest rates, power and energy markets. 
\end{abstract}


\section*{Introduction}
In this article, we study the strong approximation of the {\it alpha-CIR process}. This class of models was first studied in the context of continuous state branching processes with interaction or/and immigration, see Li and Mytnik \cite{LiMy}, Fu and Li \cite{FL} and the references within, and was recently introduced by Jiao et al. \cite{JMS}\cite{JMSS} to the mathematical finance literature to model sovereign interest rates, power and energy markets. The alpha-CIR process is an extension of the classic diffusion Cox-Ingersoll-Ross process to include jumps which are governed by a compensated spectrally positive $\alpha$-stable L\'evy process for $\alpha \in (1,2)$. More specifically, given a positive initial point $x_0$, the alpha-CIR process satisfies the following stochastic differential equation (SDE),
\begin{gather}\label{CIR_0}
dX_t = \left(a - kX_t\right)dt + \sigma_1 |X_t|^\frac{1}{2}dW_t + \sigma_2 |X_{t-}|^\frac{1}{\alpha}dZ_t,
\end{gather}
where $a$, $\sigma_1, \sigma_2$ are non-negative and $k \in \real$. The {\it diffusion} and the {\it jump coefficient} are given by $g(x) = |x|^\frac{1}{2}$ and $h(x)= |x|^\frac{1}{\alpha}$. The process $W$ is a Brownian motion and $Z$ is a compensated spectrally positive $\alpha$-stable process, independent of $W$, of the form
\begin{gather*}
Z_t = \int^t_0\int_0^\infty z \widetilde N(dz,ds),
\end{gather*}
where $\widetilde N$ is a compensated Poisson random measure with its L\'evy measure denoted by $\nu$. In other words, the process $Z$ is a L\'evy process with the characteristic triple $(0,\nu,\gamma_0)$, where $\gamma_0 := -\int^\infty_1 x \nu(dx)$. In general, under some monotonicity conditions on the jump coefficient, the above SDE will have a unique non-negative strong solution for any integrable compensated spectrally one-sided L\'evy process $Z$, see \cite{FL}\cite{LiMy}. Here we mainly focus our attention on the $\alpha$-stable case. In the special case where $\sigma_1 = 0$, the solution to \eqref{CIR_0} is termed the stable-CIR, see Li and Ma \cite{LiMa}.

The CIR and the CEV processes are theoretically non-negative and are widely used in the modelling of interest rates, default rates and volatility, e.g. Duffie et al. \cite{DFS}\cite{DPS}. Therefore, in practice, for consistency reasons the ability to simulate a positive sample path is very important. The study of positivity preserving strong approximation schemes for CIR/CEV type processes has received a great deal of attention in the literature. For the classic diffusion case we mention the works of Alfonsi \cite{Alf}\cite{Alf2}, Berkaoui et al. \cite{BBD}, Brigo and Alfonsi \cite{BA}, Dereich et al. \cite{DNS}, Neuenkirch and Szpruch \cite{NeSz} and the references within. The jump-extended case (with and without delays) has recently received increasing attention, we refer to Yang and Wang \cite{YW}, and the recent working papers of Fatemion Aghdas \cite{FHT} and Stamatiou \cite{ST}. To the best of our knowledge, for jump-extended CIR/CEV models, the existing results have all focused on the case of finite activity jumps (the jumps are governed by a Poisson process) and results on positivity preserving strong approximation schemes in the case of infinite activity jumps have yet to be obtained.

The common approach in devising a positivity preserving simulation scheme for jump-extended CIR or diffusion CIR models is to first transform the solution $X$ to remove the diffusion coefficient and then combine an existing positivity preserving scheme for the It\^o diffusions, such as the backward Euler-Maruyama scheme, with the jumps of the Poisson process to create a jump-adapted scheme. The first order convergence rate obtained in for example \cite{YW} is very attractive, however schemes of this type can suffer from high computational costs when the intensity rate is high. 

Unfortunately, the existing techniques for Poisson jumps do note translate well into the case of infinite activity. This is because, in the infinite activity case, one can not simulate individually the small jumps and transform methods will usually lead to extra jump terms, which, due to presence of the small jumps are impossible to simulate. Depending on the application, one possible alternative is to consider weak approximation schemes by using a gaussian approximation of the small jumps as done in Asmussen and Rosinski \cite{AR} or Kohatsu-Higa and Tankov \cite{KHT}.

In the case of the alpha-CIR process, one is able to obtain some results in this direction. We propose an implicit approximation scheme in \eqref{scheme}, which extends the scheme proposed in Alfonsi \cite{Alf} and Brigo and Alfonsi \cite{BA} for the classic diffusion CIR process. This method depends strongly  on the fact that the diffusion coefficient $g(x)$ is a square root, which allows one to device a positivity preserving implicit scheme by solving a quadratic equation. 

Although the derivation of the current scheme follows closely the idea presented in \cite{Alf}\cite{BA}, the inclusion of an infinite activity jump process makes the scheme behave very differently to the classic implicit scheme for the diffusion CIR model, and the proof of convergence is technically more difficult.

In the diffusion case, the discriminate of the previously mentioned quadratic equation is non-negative if the condition $a-\sigma_1^2/2 > 0$ is satisfied.
In the proposed scheme, the support of the discriminate is bounded below only in the case where $Z$ has finite activity jumps (Type A) or is of finite variation and has infinite activity jumps (Type B).
In the case where $Z$ is of infinite variation (Type C) and therefore has infinite activity jumps, there is no hope of finding a set of conditions on the parameters $a,k,\sigma_1$, $\sigma_2, \alpha$ and the grid size such that the discriminate is non-negative.
This issue is resolved by taking the absolute value of the constant term in the quadratic equation, which ensures the non-negativity of the discriminate and thus the existence of an unique positive root (Descartes' Sign Rule). 

We point out that in the diffusion case, by using more advanced techniques, c.f. Hefter and Herzwurm \cite{HH}, it is possible to relax the condition $a-\sigma_1^2/2>0$. However, it is not clear if these techniques can be translated to the jump-extended setting. Here the condition $a-\sigma_1^2/2 > 0$ is essential in controlling the probability that the discriminate is negative in Lemma \ref{dnegative}. The case $\alpha = 2$ is not studied in this work, this is because for $\alpha = 2$ the alpha-CIR can be reduced to the diffusion CIR (see Jiao et al. \cite{JMS}) and one can refer to previously mentioned works on the diffusion case.


Unfortunately, depending on the integrability of $Z$, in order to obtain strong convergence of the proposed scheme, our methodology requires one to first modify the jump coefficient and the rate of convergence is obtained in two steps. That is we first truncate and consider the bounded jump coefficient given by $h(x)	=\min\{|x|^\frac{1}{\alpha}, H\}$, for some arbitrarily large constant $H>1$, and then compute the strong rate of convergence in Theorem \ref{main_1} for the approximation scheme $X^{H,n}$ of the {\it truncated alpha-CIR process} $X^H$, i.e. the solution to \eqref{CIR_0} with jump coefficient $\min\{ |x|^\frac{1}{\alpha}, H\}$. Then we compute in Theorem \ref{main_2} the strong rate of convergence of the truncated alpha-CIR process $X^H$ towards the alpha-CIR process $X$ as $H\uparrow \infty$.  Finally, by carefully selecting $H$ as a function of the grid size, we obtain for $\sqrt{2}<\alpha <2$ the overall rate of convergence in Corollary \ref{cor2.9}.

%

To this end, it is worth pointing out that the jump coefficient is truncated for purely technical reasons. In fact, to the best of our knowledge, there are no results on Euler-Maruyama schemes for (symmetric) $\alpha$-stable processes with unbounded jump coefficient, e.g. Hashimoto \cite{Ha} and Hashimoto and Tsuchiya \cite{HaTsu}.
On the other hand, we note that truncation of the jump coefficient $h$ is not needed in the case where $Z$ is square integrable. For example, when $Z$ is compensated Poisson process or a compensated spectrally positive tempered $\alpha$-stable process for $\alpha\in (1,2)$.

Finally, we mentioned that the current problem can potentially be treated using the symmetrized Euler-Maruyame scheme studied Diop \cite{DA}, Berkaoui et al. \cite{BBD} and Bossy and Diop \cite{BD}.
However, local time techniques used in \cite{BBD} \cite{BD} \cite{DA} do not translate well to our setting. This is due to the lack of a suitable version of the It\^o-Tanaka formula for $\alpha$-stable process or in general, L\'evy processes of infinite variation.


\section*{Notations and Assumptions}
We work on a usual filtered probability space $(\Omega,\F,\mathbb{P})$ with a filtration $\ff = (\F_t)_{t\geq 0}$, and we assume that all the processes considered are adapted to this filtration. For any $0\leq s<t$, the integral with upper limit $t$ and lower limit $s$ is understood as a integral over the interval $(s,t]$. Given the terminal time $T$, we consider an equally spaced grid $\pi^n:= \{(t_0,\ldots,t_n) : 0=t_0<t_1<\cdots< t_{n}=T\}$ and set $\eta(t) := t_{i}$ for $t\in (t_i,t_{i+1}]$. In addition, for $i = 0,\dots, n-1$, we require the condition that $1+k\Delta t_i  > 0$ where $\Delta t_i = t_{i+1}- t_i$.  Given a process $X$, the negative part of $X$ is denoted by $X^- := \max\{0,-X\}$ and for $i = 0,\dots, n-1$, we set $\Delta X_{t_i} := X_{t_{i+1}} - X_{t_i}$. The L\'evy measure of $Z$ will be denoted by $\nu$ and the drift of $Z$ is denoted by $\gamma_0 =
-\int^\infty_1 x \nu(dx)$. For general results on L\'evy processes we refer to Sato \cite{SK} and Applebaum \cite{AD}. In estimation, we often use $C$, $C'$, $C''$, $C_0$ or $c$ to denote constants, which may change from line to line. Subscripts will be used to indicate dependence of the constant on other parameters.

\section{An implicit scheme for the truncated alpha-CIR process}
Given a positive initial point $x_0$, we let $h(x)	=\min\{|x|^\frac{1}{\alpha}, H\}$, for some arbitrarily large constant $H > 1$, we consider the solution to the stochastic differential equation
\begin{align}\label{SDE_0}
dX_t^{H} = \left(a - kX_t^{H}\right)dt + \sigma_1 |X_t^{H}|^\frac{1}{2}dW_t + \sigma_2 h(X_{t-}^{H})dZ_t,
\end{align}
where $a$, $\sigma_1, \sigma_2$ are non-negative parameters with $a-\sigma_1^2/2>0$, $k \in \real$, $\alpha\in (1,2)$. The process $W$ is a Brownian motion and $Z$ is a compensated spectrally positive $\alpha$-stable L\'evy process with L\'evy measure $\nu$, independent of $W$. In order to derive a positivity preserving scheme for the truncated process $X^{H}$, we take our inspiration from Alfonsi \cite{Alf} or more generally Milstein et al. \cite{MRT} by considering the implicit scheme, $\widetilde X^{H,n}_{t_0} = x_0$, for $i = 0,\dots, n-1$
\begin{align}
\Delta \widetilde{X}^{H,n}_{t_{i}} &= ( a - \sigma_1^2/2 - k\widetilde{X}^{H,n}_{t_{i+1}} )\Delta t_i + \sigma_1 |\widetilde{X}^{H,n}_{t_{i+1}}|^\frac{1}{2}\Delta W_{t_i} + \sigma_2 h(\widetilde{X}^{H,n}_{t_{i}}) \Delta Z_{t_i},\label{scheme0}
\end{align}
where the extra term $(\sigma^2_1/2)\Delta t_i$ stems from the quadratic variation of $\sqrt{X^{H}}$ and $W$. 
\begin{remark}
In equation \eqref{scheme0}, the discretisation scheme is made implicit in the drift and the diffusion coefficient, but not in the jump coefficient $h$. It is important that the scheme is not implicit in the jump coefficient. Otherwise there  will be another extra adjustment term stemming from the quadratic variation of $\sqrt{X^H}$ and $Z$, which can not be simulated unless the time and size of each individual jump can be simulated. 
\end{remark}
For every $i = 0,\dots, n-1$, by setting $x^2 = \widetilde{X}^{H,n}_{t_{i+1}}$ in \eqref{scheme0}, one can obtain the following quadratic equation in $x$,
$$
(1+k\Delta t_i)x^2 - \sigma_1 \Delta W_{t_i} x - (\widetilde{X}^{H,n}_{t_i} + (a - \sigma_1^2/2)\Delta t_i + \sigma_2 h(\widetilde{X}^{H,n}_{t_{i}}) \Delta Z_{t_i})=0,
$$
which has a unique non-negative solution if the discriminate is non-negative, or under a slightly strong condition, the process 
$$
\widetilde{X}^{H,n}_{t_i} + \left(a - \sigma_1^2/2\right)\Delta t_i + \sigma_2 h(\widetilde{X}^{H,n}_{t_{i}}) \Delta Z_{t_i},
\quad i =0,1,\dots, n-1
$$ is non-negative. However the above can be negative as the drift $\gamma_0$ of $Z$ is negative.

In the case where $Z$ is a compensated spectrally positive $\alpha$-stable process, or in general a L\'evy process of infinite variation (L\'evy process of Type C), the support of the process $Z$ is not bounded below (see Theorem 24.10 (iii) in Sato \cite{SK}). Therefore it is not possible to select parameters so that the discriminate is non-negative. To overcome this, we take the absolute value and consider the following scheme, $X^{H,n}_{t_0}:= x_0$ and for each $i =0,\dots, n-1$,
\begin{align}
X^{H,n}_{t_{i+1}}
&:= \left[\frac{\sigma_1 \Delta W_{t_i} + \sqrt{(\sigma_1 \Delta W_{t_i})^2 + 4\left(1+k\Delta t_i\right)|D_{t_{i+1}}|}}{2\left(1+k\Delta t_i\right)}\,\right]^2,\label{scheme}\\
D_{t_{i+1}}
&:=X^{H,n}_{t_i} + \left(a - \sigma_1^2/2\right)\Delta t_i + \sigma_2 h(X^{H,n}_{t_{i}}) \Delta Z_{t_i}.\nonumber
\end{align} 
The goal in the rest of this article is to derive the strong rate of convergence under the assumption that $a- \sigma^2_1/2 > 0$ and $\alpha \in (1,2)$. We must point out that $Z$ can be replaced by any compensated spectrally positive integrable L\'evy process and one can show that the scheme converges given good estimates on $\mathbb{P}[D_{t_{i+1}} < 0]$ and $\mathbb{E}[|\Delta Z_{t_i}|]$.

\begin{remark}\label{Rem_0}
The difficulty of the current work lies in that a compensated spectrally positive $\alpha$-stable process is a L\'evy process of infinite variation (Type C). In the case where $Z$ has finite activity (Type A) or has infinite activity and is of finite variation (Type B), (see Definition 11.9 in \cite{SK}), it is possible to find a set of conditions on the parameters to ensure that the process $D$ is non-negative. To see this, suppose that the support of the L\'evy measure $\nu$ contains $0$, (see page 148 of \cite{SK} for the definition and properties of the support of a measure). From Theorem 24.10 (iii) in Sato \cite{SK}, we know that the support of $\Delta Z_{t_i}$ is almost surely contained in $[\gamma_0 \Delta t_i,\infty)$, where the drift $\gamma_0$ is given by $\gamma_0 = -\int^\infty_1 x\nu(dx)$.
Therefore, we obtain 
\begin{align*}
X^{H,n}_{t_i} + \left(a - \sigma_1^2/2\right)\Delta t_i + \sigma_2 h(X^{H,n}_{t_i}) \Delta Z_{t_i} 
&\geq X^{H,n}_{t_i} + \left(a - \sigma_1^2/2 
	+
\sigma_2 |X^{H,n}_{t_i}|^\frac{1}{\alpha}\gamma_0\right)\Delta t_i.
\end{align*}

By considering the convex function $z \mapsto |{X}^{H,n}_{t_i}|^z$ over the domain $[0,1]$ and applying Jensen's inequality, we see that $|{X}^{H,n}_{t_i}|^{z} \leq \left(1-z\right) + z{X}^{H,n}_{t_i}$ for all $z$ in the interval $[0,1]$. Since $1/\alpha$ lies in this interval, this inequality holds for $z=1/\alpha$. Applying this inequality to the above expression gives, 
\begin{align*}
{X}^{H,n}_{t_i} +  	\left(a - \sigma_1^2/2 
	+
\sigma_2 |X^{H,n}_{t_i}|^\frac{1}{\alpha}\gamma_0\right)\Delta t_i 
&\geq  \left(1
	+
\sigma_2 \gamma_0\Delta t_i/\alpha\right){X}^{H,n}_{t_i}\\
& \quad  + \left(a - \sigma_1^2/2 
	+
\sigma_2\gamma_0\left(1-1/\alpha\right) \right)\Delta t_i,
\end{align*}
which is positive if $1	+ \sigma_2 \gamma_0\Delta t_i/\alpha$ and $a-\sigma_1^2/2 + \sigma_2\gamma_0(1-1/\alpha)$ are both positive.  Hence, we arrive at a set of sufficient conditions on the parameters, given by 
\begin{align}
	1+\sigma_2 \gamma_0\Delta t_i/\alpha>0
	\quad \mathrm{and} \quad
	a - \sigma_1^2/2 + \sigma_2 \gamma_0\left(1 - 1/\alpha\right)>0.\label{poissoncond}
\end{align}
If $\sigma_2$ is equal to zero, then the second condition becomes $a > \sigma_1^2/2$, which is the same as the conditions imposed on the diffusion CIR process to be strictly positive.

Therefore it is worth mentioning that if $Z$ is the compensated Poisson process, then one does not need to truncate the jump coefficient or modify the discriminate once the condition given in \eqref{poissoncond} is satisfied.
\end{remark}

\section{Strong convergence}
To show that the proposed scheme converges, we expand the implicit scheme given in \eqref{scheme} around the Euler-Maruyama scheme and then apply the Yamada-Watanabe approximation technique. More explicitly, by expanding the quadratic in \eqref{scheme}, using the identity $|D| = D + 2D^-$ and adding/subtracting the appropriate terms we obtain
\begin{align}
\Delta {X}_{t_{i}}^{H,n}
		&= (a-k_n{X}_{t_i}^{H,n})\Delta t_i+ \sigma_1 |{X}_{t_i}^{H,n}|^\frac{1}{2}\Delta W_{t_i} + \sigma_2 h(X^{H,n}_{t_{i}}) \Delta Z_{t_i} + \Delta R^n_{t_i},\label{dscheme}
\end{align}
where $k_n = k(1+kT/n)^{-1}$ and $\Delta R^n_{t_i}$ (change in the remainder process) is given by
\begin{align}
\Delta R^n_{t_i} & := -\sigma_2 h(X^{H,n}_{t_{i}}) \Delta Z_{t_i} +  \frac{\sigma_1^2}{2} \left(\frac{\Delta W_{t_i}^2}{(1+kT/n)^2} - \frac{T/n}{1+kT/n}\right)\label{R}\\
& \quad + \frac{aT}{n}\left(\frac{1}{1+kT/n} -1\right) - \sigma_1 |X_{t_i}^{H,n}|^\frac{1}{
2
} \Delta W_{t_i} + \frac{\sigma_2h(X^{H,n}_{t_{i}})\Delta Z_{t_i}}{1+kT/n}\notag\\
& \quad  + \Delta M_{t_i}^n + \frac{2}{1+k T/n}D^{-}_{t_{i+1}}.\notag
\end{align}
The term $\Delta M_{t_{i}}^n$ is given by
\begin{align*} 
\Delta M_{t_{i}}^n  & := \frac{\sigma_1\Delta W_{t_i}}{2(1+kT/n)^2}\sqrt{\sigma_1^2\Delta W_{t_i}^2 + 4(1+kT/n)|D_{t_{i+1}}|}
\end{align*}
and the fact that $\Delta M_{t_{i}}^n$ is a martingale increment in the filtration $\mathbb{F} = (\F_t)_{t\geq 0}$ can be quick checked as 
\begin{align*}
& \mathbb{E}[ \Delta M_{t_{i}}^n  | \F_{t_i} ] = \\
&\frac{\sigma_1 \sqrt{T/n}}{2(1+kT/n)^2} \mathbb{E}\Big[
\int_{-\infty}^{\infty} \frac{x}{\sqrt{2\pi}} e^{-\frac{x^2}{2}} \sqrt{\sigma_1^2(T/n)\, x^2 + 4(1+kT/n)|D_{t_{i+1}}|}\,dx \,\Big| \,\F_{t_i}  \Big],
\end{align*}
which is equal to zero. Hence we can conclude that $\mathbb{E}[ \Delta M_{t_{i}}^n |\, \mathcal{F}_{t_i} ] = 0 $. 

The next step is to compute the semimartingale decomposition of $\Delta R^n$. The second term on the right hand side of \eqref{R} is given by
\begin{align*}
	\frac{\sigma_1^2}{2} \left(\frac{\Delta W_{t_i}^2}{(1+kT/n)^2} - \frac{T/n}{1+kT/n}\right) 
	&=
		\frac{\sigma_1^2}{2}
		\left(
			\frac{\Delta W_{t_i}^2-\Delta t_i}{(1+kT/n)^2}
		\right)
		-
		\frac{k\sigma_1^2}{2(1+kT/n)^2}
		\left(
			\frac{T}{n}
		\right)^2
\end{align*}
and for the last term of \eqref{R}, we can write $D^-_{t_{i+1}} =: \Delta M^D_{t_i} + \mathbb{E}[D^-_{t_{i+1}}|\F_{t_i}]$. Finally, by collecting terms appropriately, we can express $\Delta R^n_{t_i} = \Delta N^n_{t_i} + A^n_{t_{i}}$ where
\begin{align*}
	\Delta N^n_{t_i}
	&:=
	\frac{\sigma_1^2}{2}\left(\frac{\Delta W_{t_i}^2 - \Delta t_i }{(1+kT/n)^2}\right)
	- \sigma_1 |{X}_{t_i}^{H,n}|^\frac{1}{2} \Delta W_{t_i}
	+ \Delta M^n_{t_i}
	-\frac{Tk_n\sigma_2 }{n}h(X^{H,n}_{t_{i}}) \Delta Z_{t_i}
	+\frac{2\Delta M^D_{t_i}}{1+kT/n}, \notag\\
	A^n_{t_{i}}
	&:=
	\left(\frac{T}{n}\right)^2
	\left(
			-\frac{k\sigma_1^2}{2(1+kT/n)^2}
		- ak_n
	\right)
	+ \frac{2\mathbb{E}[D^-_{t_{i+1}}|\F_{t_i}]}{(1+kT/n)}.
\end{align*}
The terms in $\Delta N^n$ are martingale differences and $A^n$ is a predictable process.

Finally, the discrete time scheme \eqref{dscheme} can be then extended to continuous time by setting $X^{H,n}_t  := \bar X^{H,n}_t + R^n_t$ where
\begin{align*}
\bar X^{H,n}_t  &:= x_0 + \int^t_0 (a-k_n X^{H,n}_{\eta(s)})ds + \sigma_1\int^t_0 |X^{H,n}_{\eta(s)}|^\frac{1}{2} dW_s+ \sigma_2\int^t_0  h(X^{H,n}_{\eta(s)})dZ_s,\\
R^n_t &:= N^n_t+ \int_{(0,t]}\frac{n}{T}A^n_{\eta(s)}\,ds.
\end{align*}
The extension of the martingale part of the remainder process to continuous time is done by setting $N_t^{{n}} = \mathbb{E}[N_{t_{i+1}}^{n}|\mathcal{F}_t]$ for $t\in (t_i,t_{i+1}]$. 

\subsection{Auxiliary estimates}
In this subsection, we present some auxiliary estimates which are needed to prove the strong convergence of the proposed scheme. 

\begin{Lem}\label{dnegative}
	For $i = 0,\dots, n-1$ and $\alpha\in (1,2)$
	\begin{gather*}
	\mathbb{P}[D_{t_{i+1}} < 0 \,|\, \F_{t_i}]
	\leq \exp(-C_{a,\alpha,\sigma_1,\sigma_2}(\Delta t_i)^{-(2-\alpha)/(\alpha-1)}),
	\end{gather*}
	where $C_{a,\alpha,\sigma_1,\sigma_2}$ is some positive constant depending on $a$, $\alpha$, $\sigma_1$ and $\sigma_2$.
\begin{proof}
See subsection \ref{d} of the Appendix.
\end{proof}
\end{Lem}

\begin{Lem}\label{l3.2}
For $\alpha \in (1,2)$, we have $\sup_{n \in \n}\max_{i= 0,\dots, n-1}\mathbb{E}[X^{H,n}_{t_{i+1}}] <\infty $.
	\begin{proof}
		Taking the expectation of the scheme given in \eqref{dscheme} to obtain, for some $C>0$,
		\begin{align*}
		\mathbb{E}[X_{t_{i+1}}^{H,n}] &= \mathbb{E}[X^{H,n}_{t_i}] + (a-k_n\mathbb{E}[{X}_{t_i}^{H,n}])\Delta t_i + \mathbb{E}[A^n_{t_{i}}]\\
		&\leq 
		C\Delta t_i
		+(1+C \Delta t_i)\mathbb{E}[X^{H,n}_{t_i}]
		+2\mathbb{E}[D^{-}_{t_{i+1}}].
		\end{align*}
		To continue the calculation, by using the fact that $D^-$ is positive, we obtain
		\begin{align*}
		\mathbb{E}[D^{-}_{t_{i+1}}|\mathcal{F}_{t_i}]
		&=- (X^{H,n}_{t_i} + (a-\sigma^2_1/2)\Delta t_i)\mathbb{E}[\mathbf{1}_{\{ D_{t_{i+1}} < 0\}}|\mathcal{F}_{t_i}] \\
		&\quad
		-\mathbb{E}[\sigma_2 h(X^{H,n}_{t_{i}}) \Delta Z_{t_i}\mathbf{1}_{\{ D_{t_{i+1}} < 0\}}|\mathcal{F}_{t_i}]\\
		& \leq  \sigma_2\mathbb{E}[|X^{H,n}_{t_i}|^\frac{1}{\alpha}|\Delta Z_{t_i}|\mathbf{1}_{\{ D_{t_{i+1}} < 0\}}|\F_{t_i}].
		\end{align*}
		By taking the expectation above, applying H\"older's inequality with 
		$1/p+1/q = 1$ where $p>1$, $q \in (1,\alpha)$, we obtain 		
		\begin{align*}
		\mathbb{E}[D^{-}_{t_{i+1}}] & \leq \sigma_2\e[|X^{H,n}_{t_i}|^\frac{q}{\alpha}\Delta Z_{t_i}|^q]^\frac{1}{q}\e[\mathbf{1}_{\{ D_{t_{i+1}} < 0\}}]^{\frac{1}{p}} \\
		&\leq  \sigma_2\mathbb{E}[|X^{H,n}_{t_i}|]^\frac{1}{\alpha}\mathbb{E}[|\Delta Z_{t_i}|^{q}]^\frac{1}{q}\mathbb{E}[\mathbf{1}_{\{ D_{t_{i+1}} < 0\}}]^\frac{1}{p}  \\
		&\leq
		C(1+\mathbb{E}[|X^{H,n}_{t_i}|])(\Delta t_i)^\frac{1}{\alpha}\mathbb{E}[\mathbf{1}_{\{ D_{t_{i+1}} < 0\}}]^\frac{1}{p}.
		\end{align*}
		where in the second inequality, we have used the fact that $\Delta Z_{t_i}$ is independent from $X^{H,n}_{t_i}$ and Jensen's inequality (as $\frac{q}{\alpha}<1$). In the last inequality, we have used the fact that $|x|^\frac{1}{\alpha} \leq (1-1/\alpha)+|x|/\alpha$. 	Therefore by using the Lemma \ref{dnegative},
		\begin{align*}
			\mathbb{E}[D^{-}_{t_{i+1}}]
			&\leq
			(1+\mathbb{E}[|X^{H,n}_{t_i}|]) \Delta t_i
			\cdot C(\Delta t_i)^{-\frac{\alpha-1}{\alpha}}
			\exp(-p^{-1}C_{a,\alpha,\sigma_1,\sigma_2}(\Delta t_i)^{-(2-\alpha)/(\alpha-1)})\\
			&\leq C'(1+\mathbb{E}[|X^{H,n}_{t_i}|]) \Delta t_i,
		\end{align*}
		for some $C'>0$. From the above we recover the recursive equation
		\begin{align*}
			1+\mathbb{E}[X_{t_{i+1}}^{H,n}]
			&\leq (1+C''\Delta t_i)
			+ (1 + C''\Delta t_i) \mathbb{E}[X^{H,n}_{t_i}]\\
			&= (1+C'' \Delta t_i) (1+\mathbb{E}[X^{H,n}_{t_i}]),
		\end{align*}
		for some $C''$. This gives $1+\mathbb{E}[X^{H,n}_{t_i}] \leq (1+x)(1+C'' \Delta t_i)^{n} \leq (1+x)e^{C''T}$.
	\end{proof}
\end{Lem}

\begin{Lem}\label{l3.1}
For all $i =0,\dots, n-1$, there exists positive constants $C$ and $p$ such that 
\begin{gather*}
\mathbb{E}[D^-_{t_{i+1}}] \leq Cn^{-\frac{1}{\alpha}}\mathbb{P}[D_{t_{i+1}} < 0]^{\frac{1}{p}}
\end{gather*}
and $1/\alpha + 1/p = 1$. 
\end{Lem}
\begin{proof}
It is sufficient to combine Lemma \ref{l3.2} and the following inequality
\begin{align*}
\mathbb{E}[D^{-}_{t_{i+1}}] & \leq  C\sigma_2(1+\mathbb{E}[|X^{H,n}_{t_i}|])\Delta t_i^\frac{1}{\alpha}\mathbb{E}[\mathbf{1}_{\{ D_{t_{i+1}} < 0\}}]^\frac{1}{p} ,
\end{align*}
where $1/\alpha + 1/p =1$.
\end{proof}

\begin{Lem}\label{lem2.4}
For $\alpha \in (1,2)$, we have $$\mathbb{E}[|\Delta M^n_{t_i} - \sigma_1 |{X}_{t_i}^{H,n}|^\frac{1}{2} \Delta W_{t_i}|^2] \leq Cn^{-({\frac{1}{\alpha}+1})}.$$
\end{Lem}
\begin{proof}
To estimate $\big| \Delta M_{t_i}^n - \sigma_1 |{X}_{t_i}^{H,n}|^\frac{1}{2} \Delta W_{t_i}\big|^2$, we proceed by setting $\kappa_n:=1+kT/n$ and we note that it can be estimated by
\begin{align*}
& \left| \frac{\sigma_1\Delta W_{t_i}}{2\kappa_n^2}\left[\sqrt{\sigma_1^2\Delta W_{t_i}^2 + 4\kappa_nD_{t_{i+1}}  + 8\kappa_n D^{-}_{t_{i+1}}}-2\kappa_n^2\sqrt{X_{t_i}^{H,n}}\,\,\right]\right|^2
\leq \nonumber  \\
& \frac{\sigma_1^2|\Delta W_{t_i}|^2}{4\kappa_n^4}
\left| \sigma_1^2\Delta W_{t_i}^2 + 4\kappa_n\big\{(1-\kappa_n^3){X}_{t_i}^{H,n} + \left( a - \sigma_1^2/2 \right)T/n + \sigma_2 h({X}_{t_i}^{H,n})\Delta Z_{t_i} +2 D^{-}_{t_{i+1}}\big\} \right|.
\end{align*}
The right hand side of the above inequality can be estimated using the fact that $|1-\kappa_n^3| \leq Cn^{-1}$, $X_{t_i}^{H}$ and $\Delta Z_{t_i}$ are independent and $\mathbb{E}[|\Delta Z_{t_i}|] \leq Cn^{-\frac{1}{\alpha}}$. Therefore, by Lemma \ref{l3.2} and Lemma \ref{l3.1}, the expectation of the right hand side above is bounded by an integrable random variable multiplied by $n^{-(\frac{1}{\alpha}+1)}$.
\end{proof}


\begin{Lem}\label{l2.4}
The remainder process satisfies $\mathbb{E}[|R^n_t|] \leq C_T n^{-\frac{1}{2\alpha}}$ or more specifically, the process $N^n$ and $A^n$ satisfies 
\begin{gather*}
\mathbb{E}[|N^n_t|] \leq C_Tn^{-\frac{1}{2\alpha}} \quad \mathrm{and} \quad \mathbb{E}\left[\left|\int_{(0,t]} \frac{n}{T} A^n_{\eta(s)} ds \right|\right] \leq C_Tn^{-1},
\end{gather*}
where $C_T$ is some constant depending on $T$.
\end{Lem}

\begin{proof}
Using Lemma \ref{dnegative}, that is $\mathbb{E}[D^-_{t_{i+1}}]$ is exponentially small with respect to the discretisation grid, we obtain 
\begin{align*}
\mathbb{E}\left[\left|\int_{(0,t]} \frac{n}{T} A^n_{\eta(s)} ds \right|\right] \leq \sum_{i=1}^n\mathbb{E}[|A^n_{t_i} |] \leq C_Tn^{-1}.
\end{align*}
Recall that the martingale differences $\Delta N^n$ can be expressed as
\begin{align*}
\Delta N^n_{t_i}
= \Delta \hat M^n_{t_i} + \Delta \bar M^n_{t_i} - \frac{T}{n}k_n\sigma_2 h(X_{t_i}^{H,n}) \Delta Z_{t_i}   + \frac{2\Delta M^D_{t_i}}{(1+kT/n)},
\end{align*}
where we set
\begin{align*}
	\Delta \hat M^n_{t_i}
	:=\frac{\sigma_1^2}{2}\left(\frac{\Delta W_{t_i}^2 - \Delta t_i }{(1+kT/n)^2}\right)
	\quad \text{and} \quad
	\Delta \bar M^n_{t_i}
	:=\Delta M_{t_i}^n - \sigma_1 |{X}_{t_i}^{H,n}|^\frac{1}{2} \Delta W_{t_i}.
\end{align*}

In order to estimate $\mathbb{E}[|N_t^n|]$, we note that for $t\in [0,T]$
\begin{align*}
N^n_t 
& = \sum_{j=0}^{n-1} \1_{(t_j,t_{j+1}]}(t) \mathbb{E}\big[ \sum_{i=0}^{j} \Delta N^n_{t_i} \big|\, \F_{t}\big].
\end{align*}
By taking the absolute value, applying the Cauchy-Schwarz inequality, Jensen's inequality and the $L^2$-isometry to the martingale increment $\Delta \hat M^n_{t_i} + \Delta \bar M^n_{t_i}$, we obtain
\begin{align*}
	\mathbb{E}[|\sum^j_{i=0} \Delta N^n_{t_i}|] 
	&
	\leq \mathbb{E}[|\sum_{i=0}^j \Delta (\hat M^n_{t_i} + \bar M^n_{t_i})|]
	+\sum_{i=0}^j\frac{|k|T\sigma_2}{n}\mathbb{E}[|X_{t_i}^{H,n}|^{\frac{1}{\alpha}}] \mathbb{E}[|\Delta Z_{t_i}|]
	+2\sum_{i=1}^j\mathbb{E}[|\Delta M^D_{t_i}|]\\
	& \leq \Big[\sum^n_{i=1} \mathbb{E}\big[\big(\Delta (\hat M^n_{t_i} + \bar M^n_{t_i})\big)^2\big]\Big]^\frac{1}{2} + Cn^{-\frac{1}{\alpha}}
	+2\sum^n_{i=1}\mathbb{E}[|\Delta M^D_{t_i}|]\\
	& \leq \Big[2\sum_{i=1}^n \big[\mathbb{E}[(\Delta \hat M^n_{t_i})^2] + \mathbb{E}[(\Delta\bar M^n_{t_i})^2]\big]\Big]^\frac{1}{2} + Cn^{-\frac{1}{\alpha}}+ n^{1-\frac{1}{\alpha}}\exp(-cn^{\frac{2-\alpha}{\alpha-1}}),
\end{align*}
where $\mathbb{E}[|\Delta M^D_{t_i}|] \leq \exp(-cn^{\frac{2-\alpha}{\alpha-1}})$ follows from Lemma \ref{dnegative} and Lemma \ref{l3.1}. On the other hand, from Lemma \ref{lem2.4}
\begin{align*}
\sum_{i=0}^n \big[\mathbb{E}[(\Delta \hat M_{t_i})^2] + \mathbb{E}[(\Delta\bar M_{t_i})^2]\big]   \leq C n(n^{-2} + n^{-(\frac{1}{\alpha}+1)}) = C(n^{-1} + n^{-\frac{1}{\alpha}})
\end{align*}
and the result follows by taking the square root.
\end{proof}
\begin{Lem}\label{lem2}
Let $X^{H,n}$ be the scheme defined in \eqref{scheme} then
\begin{align*}
\sup_{t\leq T}\mathbb{E}[|X^{H,n}_{t} - X^{H,n}_{\eta(t)}|] \leq C_Tn^{-\frac{1}{\alpha}}.
\end{align*}
\begin{proof}
By taking the absolute value, the expectation and using independence increments property, for every $i = 0,\dots, n-1$ we obtain for $t\in (t_i,t_{i+1}]$
\begin{align*}
\mathbb{E}[|X_{t}^{H,n}- X_{\eta(t)}^{H,n}|] &\leq  (a+|k|\mathbb{E}[{X}_{t_i}^{H,n}])T/n+ \sigma_1 \mathbb{E}[|{X}_{t_i}^{H,n}|^\frac{1}{2}](T/n)^\frac{1}{2}\\
& \quad + C\sigma_2 \mathbb{E}[|{X}_{t_i}^{H,n}|^\frac{1}{\alpha}](T/n)^\frac{1}{\alpha}+ \mathbb{E}[|\Delta N^n_{t_i}|]+ \mathbb{E}{[|A^n_{t_{i}}|]} \leq  C_T n^{-\frac{1}{\alpha}},
\end{align*}
where we have applied Lemma \ref{l3.2} and Lemma \ref{l2.4} in the last line.
%
\end{proof}
\end{Lem}

\begin{remark}\label{rem2.1}
In the case $\sigma_1 = 0$, the estimates obtained in Lemma \ref{dnegative}, Lemma \ref{l3.2}, Lemma \ref{l3.1}, Lemma \ref{l2.4} and Lemma \ref{lem2} remains the same, and the term considered in Lemma \ref{lem2.4} is zero.
\end{remark}

\subsection{Strong rate of convergence for the truncated alpha-CIR}
We present in the following the strong rate of convergence of the positivity preserving approximation scheme $X^{H,n}$ to the truncated alpha-CIR process $X^{H}$ as $n\uparrow\infty$. The proof relies on a careful application of the Yamada-Watanabe approximation technique. For readers who are unfamiliar with the approximation technique, we included some useful results in section \ref{yamada} of the appendix.

Before proceeding, we point out that the Yamada-Watanabe approximation technique can not be applied directly to $|X^{H} - X^{H,n}|$ as done in Alfonsi \cite{Alf}. This is due to the presence of the remainder process $R^n$. The differences/difficulties that we faced here are the following, (i) the martingale representation property do not hold and one can not hope to obtain estimates from an direct application of the It\^o formula for L\'evy processes, see for example Applebaum \cite{AD}, as the explicit form of the martingale $M^D$ is not known. (ii) Even in the case where $M^D$ can be computed or $Z$ is square integrable and therefore martingale representation property holds, the monotone increasing condition on the jump coefficients (which is crucial in the proof of strong uniqueness in \cite{LiMy} and \cite{FL}) may not be satisfied. To overcome the above mentioned issues, we notice that the strong error can be decomposed into $|X_t^{H} - X^{H,n}_t| \leq | X_t^{H} - \bar X^{H,n}_t| + |R^n_t|$. The estimate of the process $R^n$ is readily available in Lemma \ref{l2.4}, and we need only to apply the Yamada-Watanabe approximation technique to the process $Y^{H,n} := X^{H}- \bar X^{H,n}$. 

\begin{Thm}\label{main_1}
Let $\sigma_1> 0$, $\sigma_2 > 0$ then the strong rate of convergence is given by
\begin{gather*}
\sup_{t\leq T}\mathbb{E}[|X_t^{H} - X^{H,n}_t|]
\leq
C\{ (\log n)^{-1}+H(\log n)^{\alpha-1}n^{-\frac{1}{4\alpha^2}}\}.
\end{gather*}
Let $\sigma_1= 0$, $\sigma_2 > 0$ then the strong rate of convergence is given by
\begin{gather*}
\sup_{t\leq T}\mathbb{E}[|X_t^{H} - X^{H,n}_t|] \leq CH n^{-(\frac{2}{\alpha}-1)\frac{1}{4\alpha}}.
\end{gather*}
\end{Thm}

\begin{proof}
Let $\varepsilon \in (0,1)$ and $\delta \in (1,\infty)$. We define ${Y}^{H,n}:=X^{H}-\bar X^{H,n}$ and denote the jumps of $Y^{H,n}$ by $\Delta Y^{H,n}_s(z) := \sigma_2 \{h(X_{s}^{H})-h(X_{\eta(s)}^{H,n})\}z$. By applying It\^o's formula to the Yamada-Watanabe function $\phi_{\delta,\epsilon}(Y^{H,n})$, see equation \eqref{yamada1}, we obtain
	\begin{align*}
	&|{Y}_t^{H,n}|
	\leq \varepsilon
	+\phi_{\delta,\varepsilon}({Y}_t^{H,n})
	=\varepsilon
	+{M}_t^{n,\delta,\varepsilon}
	+{I}_t^{n,\delta,\varepsilon}
	+{J}_t^{n,\delta,\varepsilon}
	+{K}_t^{n,\delta,\varepsilon},
	\end{align*}
where we have set
	\begin{align*}
	{M}_t^{n,\delta,\varepsilon}
	:=& \sigma_1 \int_{0}^{t} \phi_{\delta,\varepsilon}'(Y_{s-}^{H,n})\{|X_{s}^{H}|^\frac{1}{2}-|X_{\eta(s)}^{H,n}|^\frac{1}{2}\}dW_{s}\\
	&+\int_{0}^{t} \int_{0}^{\infty}
	\left\{
	\phi_{\delta,\varepsilon}(Y_{s-}^{H,n}+\Delta Y^{H,n}_s(z))-\phi_{\delta,\varepsilon}(Y_{s-}^{H,n})
	\right\}
	{\widetilde N}(ds,dz),\\
	{I}_t^{n,\delta,\varepsilon}
	:=&\int_{0}^{t} \phi_{\delta,\varepsilon}'(Y_{s-}^{H,n})\{-k_nX_{s}^{H} + k_nX_{\eta(s)}^{H,n}\}ds,
	\quad\\
	{J}_t^{n,\delta,\varepsilon}
	:=&\frac{\sigma_1^2}{2} \int_{0}^{t} \phi_{\delta,\varepsilon}''(Y_{s}^{H,n}) ||X_{s}^{H}|^\frac{1}{2}-|X_{\eta(s)}^{H,n}|^\frac{1}{2}|^2ds,\\
	{K}_t^{n,\delta,\varepsilon}
	:=&
	\int_{0}^{t} \int_{0}^{\infty}
	\Big\{
	\phi_{\delta,\varepsilon}(Y_{s-}^{H,n}+\Delta Y^{H,n}_s(z))-\phi_{\delta,\varepsilon}(Y_{s-}^{H,n})-\Delta Y^{H,n}_s(z) \phi_{\delta,\varepsilon}'(Y_{s-}^{H,n})
	\Big\}
	\nu(dz)ds.
	\end{align*}
The local martingale term $M^{n,\delta. \epsilon}$ will disappear after applying the standard localization argument and taking the expectation. In the following, we suppose that all integrals are appropriately stopped, and to simplify notation, we do not explicitly write the localizing sequence of stopping times and will focus on obtaining upper estimates of ${K}_t^{n,\delta,\varepsilon}$, ${I}_t^{n,\delta,\varepsilon}$ and ${J}_t^{n,\delta,\varepsilon}$ which are independent of the localizing sequence.
%
%

To estimate $K_t^{n,\delta,\varepsilon}$, we write it into two terms
$
	K_t^{n,\delta,\varepsilon}
	=K_t^{n,\delta,\varepsilon,1}+K_t^{n,\delta,\varepsilon,2}
$ where
	\begin{align*}
	K_t^{n,\delta,\varepsilon,1}
	&:=
	\int_{0}^{t} \int_{0}^{\infty}
	\Big\{
	\phi_{\delta,\varepsilon}(Y_s^{H,n}+\sigma_2 \{h(X_s^{H})-h(\bar X^{H,n}_s)\}z)-\phi_{\delta,\varepsilon}(Y_s^{H,n})\\
	&\quad 	-\sigma_2 \{h(X_s^{H})-h(\bar X^{H,n}_s)\}z \phi_{\delta,\varepsilon}'(Y_{s}^{H,n})
	\Big\}
	\nu(dz)ds\\
	K_t^{n,\delta,\varepsilon,2}
	&:=
	\int_{0}^{t} \int_{0}^{\infty}
	\Big\{
	\phi_{\delta,\varepsilon}(Y_s^{H,n}+\sigma_2 \{h(X_s^{H})-h(X_{\eta(s)}^{H,n})\}z)-\phi_{\delta,\varepsilon}(Y_s^{H,n}+\sigma_2\{h(X_{s}^{H})-h(\bar X^{H,n}_s)\}z)\\
	& \quad -\sigma_2\{h(\bar X^{H,n}_s)-h(X_{\eta(s)}^{H,n})\}z \phi_{\delta,\varepsilon}'(Y_s^{H,n})
	\Big\}
	\nu(dz)ds.
	\end{align*}

	We observe that if $Y^{H,n} = X^{H}-\bar X^{H,n}=0$ then $h(X^{H})-h(\bar X^{H,n})=0$. Therefore we can apply Lemma \ref{key_lem_0} with $y=Y_s^{H,n}$ and $x=\sigma_2\{h(X_s^{H})-h(\bar X^{H,n}_s)\}$, since $h(x) = \min\{|x|^\frac{1}{\alpha},H\}$ is non-decreasing. We obtain for any $u>0$, 
	\begin{align}
	&\int_{0}^{\infty}
	\left\{
	\phi_{\delta,\varepsilon}(Y_s^{H,n}+\sigma_2\{h(X_s^{H})-h(\bar X^{H,n}_s)\}z)-\phi_{\delta,\varepsilon}(Y_s^{H,n})
	-\sigma_2\{h(X_s^{H})-h(\bar X^{H,n}_s)\}z \phi_{\delta,\varepsilon}(Y_s^{H,n})
	\right\}
	\nu(dz) \notag\\
	&\leq
	\frac{2\sigma_2^2|h(X_s^{H})-h(\bar X^{H,n}_s)|^2 \1_{(0,\varepsilon]}(|Y_s^{H,n}|)}{|Y_s^{H,n}| \log \delta} \int_{0}^{u} z^2 \nu(dz)
	+ 2\sigma_2|h(X_s^{H})- h(\bar X^{H,n}_s)| \1_{(0,\varepsilon]}(|Y_s^{H,n}|) \int_{u}^{\infty} z \nu(dz)\notag\\
	& \leq \frac{2\sigma_2^{2} |Y_s^{H,n}|^{\frac{2}{\alpha}} \1_{(0,\varepsilon]}(|Y_s^{H,n}|)}{|Y_s^{H,n}| \log \delta} \int_{0}^{u} z^2 \nu(dz)
	+ 2\sigma_2 |Y_s^{H,n}|^{\frac{1}{\alpha}} \1_{(0,\varepsilon]}(|Y_s^{H,n}|) \int_{u}^{\infty} z \nu(dz) \notag \\
	& \leq \frac{2\sigma_2^{2}}{\log \delta} \varepsilon^{\frac{2}{\alpha}-1} \int_{0}^{u} z^2 \nu(dz)
	+ 2\sigma_2 \varepsilon^{\frac{1}{\alpha}} \int_{u}^{\infty} z \nu(dz),\label{EP_2}
	\end{align}
where in the second last inequality, we used the fact that $h$ is $1/\alpha$-H\"older continuous.

	By applying \eqref{key_lem_122} in Lemma \ref{key_lem12}, with $u=1,
	y=Y_{s}^{H,n},
	x=\sigma_2\{h(X_{s}^{H})-h(X_{\eta(s)}^{H,n})\}$ and $x'=\sigma_2\{h(X_{s}^{H})-h(\bar X_{s}^{H,n})\}$, and the fact that $h$ is bounded, $K_t^{n,\delta,\varepsilon,2}$ can be bounded as follows
	\begin{align}
	&K_t^{n,\delta,\varepsilon,2}
	\leq |K_t^{n,\delta,\varepsilon,2}| \label{EP_7.1}\\
	&\leq 2 \int_{0}^{1} z^2 \nu(dz) \int_{0}^{t}
	 \frac{\delta}{\varepsilon \log \delta} 
		\left(
			\sigma_2^2|h(\bar X_{s}^{H,n})-h(X_{\eta(s)}^{H,n})|^2\right. \notag\\
	& \quad +\left.\sigma_2^2|h(X_{s}^{H})-h(\bar X_{s}^{H,n})| |h(\bar X_{s}^{H,n})-h(X_{\eta(s)}^{H,n})|
		\right)
		\,ds \nonumber \\
	&	\quad + 2\sigma_2 \int_{1}^{\infty} z \nu(dz) \int_{0}^{t} |h(\bar X_{s}^{H,n})-h(X_{\eta(s)}^{H,n})| ds \nonumber \\
	&\leq 2
	\left\{
		\frac{4\sigma_2^2H\delta}{\varepsilon \log \delta}
		\int_{0}^{1} z^2 \nu(dz) \,ds
		+ \sigma_2\int_{1}^{\infty} z \nu(dz)
	\right\}
	\int_{0}^{t}
	|h(\bar X_{s}^{H,n})-h(X_{\eta(s)}^{H,n})| ds \nonumber \\
	&\leq 2
	\left\{
		\left( 4\sigma_2\int_{0}^{1} z^2 \nu(dz)\right) \vee \sigma_2\int_{1}^{\infty} z \nu(dz)
	\right\}
	\left(\frac{H\delta}{\varepsilon \log \delta} + 1\right)
	\int_{0}^{t} |\bar X_{s}^{H,n}- X_{\eta(s)}^{H,n}|^{\frac{1}{\alpha}} ds.\nonumber 
	\end{align}
By writing $X^{H,n}_t -X^{H,n}_{\eta(t)}= \bar X^{H,n}_t- X^{H,n}_{\eta(t)} + R^n_t$, from Lemma \ref{l2.4} and Lemma \ref{lem2}
\begin{gather}
\mathbb{E}[|\bar X^{H,n}_t- X^{H,n}_{\eta(t)}|] \leq \mathbb{E}[|X^{H,n}_t -X^{H,n}_{\eta(t)}|] + \mathbb{E}[|R^n_t|] \leq  C_T\{n^{-\frac{1}{\alpha}} + n^{-\frac{1}{2\alpha}}\}.\label{e5}
\end{gather}
Therefore, by taking the expectation of the both hand sides of inequality \eqref{EP_7.1}, we obtain for some $C>0$,
\begin{align}
	\mathbb{E}[K_t^{n,\delta,\varepsilon,2}]
	\leq
	CT\left(1+\frac{H\delta}{\varepsilon \log \delta}\right) n^{-\frac{1}{2\alpha^2}}\label{EP_8}.
\end{align}
To estimate the term ${I}_t^{n,\delta,\varepsilon}$, we again apply Lemma \ref{lem2} to obtain
\begin{align*}
	\mathbb{E}[|{I}_t^{n,\delta,\varepsilon}|]
	& \leq C\mathbb{E}[\int_{0}^{t} |\phi_{\delta,\varepsilon}'(Y_{s-}^{H,n})||X_{s}^{H} -X_{\eta(s)}^{H,n}|ds]\\
	& \leq C\int_{0}^{t} \mathbb{E}[|X_{s}^{H} -X_{\eta(s)}^{H,n}|] ds \leq C\int_{0}^{t} \mathbb{E}[|X_{s}^{H} -X_{s}^{H,n}|] ds  + C_Tn^{-\frac{1}{\alpha}}.
\end{align*}
To estimate the term $J^{n,\delta, \epsilon}$, we write
\begin{align*}
{J}_t^{n,\delta,\varepsilon}
	\leq & \frac{\sigma_1^2}{2} \int_{0}^{t} \phi_{\delta,\varepsilon}''(Y_{s}^{H,n}) ||X_{s}^{H}|^\frac{1}{2}-|X_{\eta(s)}^{H,n}|^\frac{1}{2}|^2ds,\\
	\leq & \frac{\sigma_1^2}{2}\left\{ \int_{0}^{t} \phi_{\delta,\varepsilon}''(|Y_{s}^{H,n}|) |Y^{H,n}_{s} | ds + \frac{1}{2} \frac{2\delta }{\epsilon \log \delta} \int_{0}^{t} |\bar X^{H,n}_{s} - X_{\eta(s)}^{H,n}|ds \right\}\\
	\leq & \frac{\sigma_1^2 T}{\log\delta} + \frac{\sigma_1^2 \delta}{\epsilon \log \delta} \int_{0}^{t} |X^{H,n}_{s} - X_{\eta(s)}^{H,n}|ds  + \frac{\sigma_1^2 \delta}{\epsilon \log \delta}\int^t_0 |R_s^n| ds 
\end{align*}
which by Lemma \ref{lem2} or \eqref{e5} shows that 
\begin{align*}
\mathbb{E}[|{J}_t^{n,\delta,\varepsilon}|] 
& \leq 
C_T \sigma^2_1\left\{
	\frac{1}{\log\delta} +  \frac{\delta}{\epsilon \log \delta} n^{-\frac{1}{2\alpha}}
\right\}.
\end{align*}
Putting this together, we obtain 
\begin{align*}
\mathbb{E}[|X_t^{H} -X^{H,n}_t|] 
		& \leq
		\varepsilon
		+\mathbb{E}[|R^n_t|]
		+C_T \left\{
		\frac{\varepsilon^{\frac{2}{\alpha}-1}}{\log \delta}
		+ \varepsilon^{\frac{1}{\alpha}}
		+ \left(1+\frac{H \delta}{\varepsilon \log \delta}\right)n^{-\frac{1}{2\alpha^2}}\right.\\
		& \quad + \left.\int_{0}^{t} \mathbb{E}[|X_{s}^{H} -X_{s}^{H,n}|] ds 
		+n^{-\frac{1}{\alpha}}
		+\sigma_1\left(\frac{1}{\log\delta}
		+ \frac{\delta}{\epsilon \log \delta} n^{-\frac{1}{2\alpha}}\right)
		\right\}.
\end{align*}
Recall that $\mathbb{E}[|R_n|] \leq C_Tn^{-\frac{1}{2\alpha}}$ and we choose $\varepsilon=(\log n)^{-\alpha}$ and $\delta=n^{\frac{1}{4\alpha^2}}$. Then by Gronwall's inequality we conclude that
\begin{align*}
	\mathbb{E}[|X_t^{H} -X^{H,n}_t|] 
	\leq C\{ (\log n)^{-1}+H(\log n)^{\alpha-1}n^{-\frac{1}{4\alpha^2}}\}.
\end{align*}
In the case $\sigma_1=0$, we choose $\delta = 2$ and we have
\begin{align*}
\mathbb{E}[|X_t^{H} -X^{H,n}_t|] 
		& \leq
		\varepsilon
		+\mathbb{E}[|R^n_t|]
		+C_T \left\{
		\frac{\varepsilon^{\frac{2}{\alpha}-1}}{\log 2}
		+ \varepsilon^{\frac{1}{\alpha}}
		+ \left(1+\frac{2H}{\varepsilon \log 2}\right)n^{-\frac{1}{2\alpha^2}}\right.\\
		& \quad + \left.\int_{0}^{t} \mathbb{E}[|X_{s}^{H} -X_{s}^{H,n}|] ds 
		+ n^{-\frac{1}{\alpha}}
		\right\}.
\end{align*}
Then by using Gronwall's inequality, the fact that $0<\frac{2}{\alpha} - 1  < \frac{1}{\alpha}< 1$ and choosing $\epsilon = n^{-p}$ where $p = \frac{1}{4\alpha}$, we obtain
\begin{align*}
	\mathbb{E}[|X_t^{H} -X^{H,n}_t|] 
	& \leq C\left\{
		\varepsilon^{\frac{2}{\alpha}-1}
		+ \left(1+\frac{2H}{\varepsilon \log 2}\right){n^{-\frac{1}{2\alpha^2}}}
		\right\} \\
		& \leq CHn^{-(\frac{2}{\alpha}-1)\frac{1}{4\alpha}}.
\end{align*}

It is clear that the upper bound is independent of the localizing sequence of stopping times, the final result then holds by Fatou's Lemma.
\end{proof}

\vskip10pt
\begin{remark}
The condition that the jump coefficient $h$ is bounded is only used in the estimation of $K^{n,\delta, \epsilon,2}$ in equation \eqref{EP_8}. The boundedness condition is used to make up for the lack of knowledge on the integrability of $X^{H,n}$ for $H \in \mathbb{R}_+\cup \{\infty\}$. We claim that if one can show for $\beta\in (1,\alpha)$, that $\sup_{n}\mathbb{E}[|X^{H,n}_t|^\beta]<\infty$ for $H = \infty$ then the strong rate of convergence can be obtained for certain range of values of $\alpha$ without first having to truncate the jump coefficient. However, it appears that uniform bound on the $\beta$-th moment of the approximation process $X^{H,n}$ is difficult to obtain, however for the theoretical process $X^{H}$ and $X$, one can show that for all $\beta \in [1,\alpha)$ we have $\mathbb{E}[|X_t^{H}|^\beta]<\infty$ and $\mathbb{E}[|X_t|^\beta]<\infty$. This result is given in Lemma \ref{alpha_norm_0} and is needed in the next subsection to study the convergence of the truncated alpha-CIR $X^{H}$ towards the alpha-CIR process $X$ as $H\uparrow \infty$.
\end{remark}

\vfill\break
\subsection{Convergence of the truncated alpha-CIR to the alpha-CIR}
%

\begin{Thm}\label{main_2}
	Assume that $\alpha \in (\sqrt{2},2)$, $\sigma_1 > 0$ and $\sigma_2 >0$ then there exists $C>0$ such that for any $H>1$,
	\begin{align*}
		\sup_{t\leq T}
		\e[|X_t-X_t^H|]
		\leq C(\log H)^{-1}.
	\end{align*}
Assume that $\alpha \in (\sqrt{2},2)$, $\sigma_1 = 0$ and $\sigma_2 >0$ then for any $p \in (0,\alpha^2-2)$, there exists $C_{T,p}>0$ such that for any $H>1$,	
	\begin{align*}
		\sup_{t\leq T}
		\e[|X_t-X_t^H|]
		\leq
		C_{T,p}
		H^{-p(1-\alpha/2)},
	\end{align*}
where $C_{T,p}$ converges to infinity as $p \uparrow (\alpha^2 - 2)$.
\end{Thm}
\begin{proof}
	Let $\varepsilon \in (0,1)$ and $\delta \in (1,\infty)$.
	We define $Y^{H}:=X-X^{H}$ and denote the jumps of $Y^{H}$ by $\Delta Y_s^{H}(z) := \sigma_2 \{|X_{s}|^{\frac{1}{\alpha}}-h(X_{s}^{H})\}z$.
	By a similar arguments of the proof of Theorem \ref{main_1}, we obtain
	\begin{align*}
	&\e[|X_t-X_t^H|]
	\leq \varepsilon
	+\e[\phi_{\delta,\varepsilon}({Y}_t^{H})]
	=\varepsilon
	+\e[{I}_t^{\delta,\varepsilon}]
	+\e[{J}_t^{\delta,\varepsilon}]
	+\e[{K}_t^{\delta,\varepsilon}],
	\end{align*}
	where we have set
	\begin{align*}
	{I}_t^{\delta,\varepsilon}
	:=&\int_{0}^{t} \phi_{\delta,\varepsilon}'(Y_{s-}^{H}) (-kX_{s}+kX_{s}^{H}) ds,
	\quad
	{J}_t^{\delta,\varepsilon}
	:=\frac{\sigma_1^2}{2} \int_{0}^{t} \phi_{\delta,\varepsilon}''(Y_{s}^{H}) ||X_{s}|^{\frac{1}{2}}-|X_{s}^{H}|^{\frac{1}{2}}|^2ds,\\
	{K}_t^{\delta,\varepsilon}
	:=&
	\int_{0}^{t} \int_{0}^{\infty}
	\Big\{
	\phi_{\delta,\varepsilon}(Y_{s}^{H}+\Delta Y_s^{H}(z))-\phi_{\delta,\varepsilon}(Y_{s}^{H})-\Delta Y_s^{H}(z) \phi_{\delta,\varepsilon}'(Y_{s}^{H})
	\Big\}
	\nu(dz)ds.
	\end{align*}
	
	To estimate ${K}_t^{\delta,\varepsilon}$, we write it into two terms ${K}_t^{\delta,\varepsilon}={K}_t^{\delta,\varepsilon,1}+{K}_t^{\delta,\varepsilon,2}$ where
	\begin{align*}
		{K}_t^{\delta,\varepsilon,1}
		:=&\int_{0}^{t} \int_{0}^{\infty}
		\Big\{
		\phi_{\delta,\varepsilon}(Y_{s}^{H}+\sigma_2\{|X_s|^{\frac{1}{\alpha}}-|X_s^{H}|^{\frac{1}{\alpha}}\}z)-\phi_{\delta,\varepsilon}(Y_{s}^{H})\\&-\sigma_2\{X_s-X_s^{H}\}z \phi_{\delta,\varepsilon}'(Y_{s}^{H})
		\Big\}
		\nu(dz)ds\\
		{K}_t^{\delta,\varepsilon,2}
		:=&\int_{0}^{t} \int_{0}^{\infty}
		\Big\{
		\phi_{\delta,\varepsilon}(Y_{s}^{H}+\sigma_2\{|X_s|^{\frac{1}{\alpha}}-h(X_s^{H})\}z)
		-\phi_{\delta,\varepsilon}(Y_{s}^{H}+\sigma_2\{|X_s|^{\frac{1}{\alpha}}-|X_s^{H}|^{\frac{1}{\alpha}}\}z)\\&
		-\sigma_2\{X_s^{H}-h(X_s^{H})\} z \phi_{\delta,\varepsilon}'(Y_{s}^{H})
		\Big\}
		\nu(dz)ds.
	\end{align*}
	To estimate ${K}_t^{\delta,\varepsilon,1}$, apply Lemma \ref{key_lem_0} with $y=Y_{s}^{H}$, $x=\sigma_2 \{|X_s|^{\frac{1}{\alpha}}-h(X_s^{H})\}$ and $u=1$, 
	\begin{align}\label{EP_2_1}
		&\int_{0}^{\infty}
		\Big\{
		\phi_{\delta,\varepsilon}(Y_{s}^{H}+\sigma_2\{|X_s|^{\frac{1}{\alpha}}-|X_s^{H}|^{\frac{1}{\alpha}}\}z)-\phi_{\delta,\varepsilon}(Y_{s}^{H})-\sigma_2\{X_s-X_s^{H}\}z \phi_{\delta,\varepsilon}'(Y_{s}^{H})
		\Big\}
		\nu(dz) \notag\\
		&\leq
		\frac{2 \sigma_2^2 ||X_s|^{\frac{1}{\alpha}}-|X_s^{H}|^{\frac{1}{\alpha}}|^{2} \1_{(0,\varepsilon]}(|Y_{s}^{H}|) }{|Y_{s}^{H}| \log \delta}
		\int_{0}^{1}
			z^2
		\nu (dz)
		+2 \sigma_2 ||X_s|^{\frac{1}{\alpha}}-|X_s^{H}|^{\frac{1}{\alpha}}| \1_{(0,\varepsilon]}(|Y_{s}^{H}|) 
		\int_{1}^{\infty}
			z
		\nu (dz) \notag\\
		&\leq
		\frac{2 \sigma_2^2 \varepsilon^{\frac{2}{\alpha}-1} }{\log \delta}
		\int_{0}^{1}
		z^2
		\nu (dz)
		+2 \sigma_2 \varepsilon^{\frac{1}{\alpha}}
		\int_{1}^{\infty}
		z
		\nu (dz).
	\end{align}
	
	By applying \eqref{key_lem_122} in Lemma \ref{key_lem12}, with $u=1,
	y=Y_{s}^{H},
	x=\sigma_2\{|X_{s}|^{\frac{1}{\alpha}}-h(X_{s}^{H})\}$ and $x'=\sigma_2\{|X_{s}|^{\frac{1}{\alpha}}-|X_{s}^{H}|^{\frac{1}{\alpha}}\}$, $K_t^{\delta,\varepsilon,2}$ can be bounded as follows
	\begin{align*}
	K_t^{\delta,\varepsilon,2}
	& \leq |K_t^{\delta,\varepsilon,2}| \notag\\
	&\leq 2 \int_{0}^{1} z^2 \nu(dz) \int_{0}^{t}
	\frac{\delta}{\varepsilon \log \delta} 
	\left\{
	\sigma_2^2||X_{s}^{H}|^{\frac{1}{\alpha}}-h(X_{s}^{H})|^2 \right. \notag\\
	& \quad + 
	\left.\,\sigma_2^2
	||X_{s}|^{\frac{1}{\alpha}}-|X_{s}^{H}|^{\frac{1}{\alpha}}|\cdot
	||X_{s}^{H}|^{\frac{1}{\alpha}}-h(X_{s}^{H})|
	\right\}
	\,ds \\
	& \quad + 2\sigma_2 \int_{1}^{\infty} z \nu(dz) \int_{0}^{t} ||X_{s}^{H}|^{\frac{1}{\alpha}}-h(X_{s}^{H})|^2 ds. \nonumber
	\end{align*}
	The assumption $\alpha \in (\sqrt{2},2)$ implies for any $p \in (0,\alpha^2-2)$ we have $\frac{p+2}{\alpha}<\alpha$.
	From Lemma \ref{alpha_norm_0}, $X_t$ and $X_t^H$ has $\frac{p+2}{\alpha}$-th moment, thus we have
	\begin{align*}
		&\e[||X_{s}^{H}|^{\frac{1}{\alpha}}-h(X_{s}^{H})|^2]
		+\e[||X_{s}|^{\frac{1}{\alpha}}-|X_{s}^{H}|^{\frac{1}{\alpha}}|\cdot
		||X_{s}^{H}|^{\frac{1}{\alpha}}-h(X_{s}^{H})|]\\
		&=\e[||X_{s}^{H}|^{\frac{1}{\alpha}}-H|^2 \1_{\{|X_s^{H}|^{\frac{1}{\alpha}}\geq H\}}]
		+\e[||X_{s}|^{\frac{1}{\alpha}}-|X_{s}^{H}|^{\frac{1}{\alpha}}|\cdot
		||X_{s}^{H}|^{\frac{1}{\alpha}}-H|\1_{\{|X_s^{H}|^\frac{1}{\alpha}\geq H\}}]\\
		&\leq
		\e[|X_{s}^{H}|^{\frac{2}{\alpha}} \1_{\{|X_s^{H}|^{\frac{1}{\alpha}}\geq H\}}]
		+\e[(||X_{s}|^{\frac{1}{\alpha}}|X_{s}^{H}|^{\frac{1}{\alpha}}+|X_{s}^{H}|^{\frac{2}{\alpha}}|) \1_{\{|X_s^{H}|^{\frac{1}{\alpha}}\geq H\}}]
		\\
		&\leq
			2H^{-p}\e[|X_{s}^{H}|^{\frac{p+2}{\alpha}}]
			+
			H^{-p}\e[|X_{s}|^{\frac{1}{\alpha}} |X_{s}^{H}|^{\frac{p+1}{\alpha}}]
		\\
		& \leq
			2H^{-p}\e[|X_{s}^{H}|^{\frac{p+2}{\alpha}}]
			+
			H^{-p}
			\e[|X_{s}|^{\frac{p+2}{\alpha}}]^{\frac{1}{p+2}}
			\e[ |X_{s}^{H}|^{\frac{p+2}{\alpha}}]^{\frac{p+1}{p+2}}
		\\
		&\leq
		C_{T,p}{H^{-p}},
	\end{align*}
	for some $C_{T,p}$ which explodes as $p\uparrow (\alpha^2-2)$ (see \eqref{l4.1const} in Lemma \ref{alpha_norm_0}).
	Therefore, $K_t^{\delta,\varepsilon,2}$ is bounded by
	\begin{align}
	&2
	\left\{
	\left( 2\sigma_2\int_{0}^{1} z^2 \nu(dz)\right) \vee \sigma_2\int_{1}^{\infty} z \nu(dz)
	\right\}
	\left(\frac{\delta}{\varepsilon \log \delta} + 1\right)
		C_{T,p}{H^{-p}}.
	\label{EP_8_1}
	\end{align}
		
	To estimate ${I}_t^{\delta,\varepsilon}$ and ${J}_t^{\delta,\varepsilon}$, we write
	\begin{align*}
	{I}_t^{\delta,\varepsilon}
	\leq k\int_{0}^{t} |X_s-X_s^{H}| ds
	\quad
	\text{and}
	\quad
	{J}_t^{\delta,\varepsilon}
	\leq & \frac{\sigma_1^2}{2} \int_{0}^{t} \phi_{\delta,\varepsilon}''(Y_{s}^{H}) ||Y_{s}^{H}|ds
	\leq \frac{\sigma_1^2}{\log \delta}.
	\end{align*}
		Putting this together, we obtain 
	\begin{align*}
	\e[|X_t-X_t^H|]
	& \leq
	\varepsilon
	+C_{T,p} \Big\{
	\frac{\varepsilon^{\frac{2}{\alpha}-1}}{\log \delta}
	+\varepsilon^{\frac{1}{\alpha}}
	+ \left(1+\frac{\delta}{\varepsilon \log \delta}\right)
		H^{-p}
	\\
	& \quad + \int_{0}^{t} \mathbb{E}[|X_{s} -X_{s}^{H}|] ds
	+\frac{\sigma^2_1}{\log\delta}
	\Big\}.
	\end{align*}
	By choosing $\varepsilon=(\log H)^{-1}$ and $\delta=H^{-p/2}$,
	and using Gronwall's inequality,
	\begin{align*}
		\sup_{t\leq T}
		\e[|X_t-X_t^H|]
		\leq C_{T,p}(\log H)^{-1}.
	\end{align*}
	Similarly, for $\alpha \in (\sqrt{2},2)$, $\sigma_1 = 0$, we choose $\delta = 2$, $\epsilon = H^{-q}$ where $q =\alpha p/2$,
	\begin{align*}
	\e[|X_t-X_t^H|]
	& \leq C_{T,p}\Big\{
	\varepsilon^{\frac{2}{\alpha}-1}
	+ \left(1+\frac{2}{\varepsilon \log 2}\right)
		H^{-p}
	\Big\}\\
	& \leq
	C_{T,p}
		H^{-p(1-\alpha/2)},
	\end{align*}
	where again the constant $C_{T,p}$ explodes as $p \uparrow (\alpha^2-2)$.
	It is clear that the upper bound is independent of the localizing sequence of stopping times, the final result then holds by Fatou's Lemma. 
\end{proof}
%

\begin{Cor}\label{cor2.9}
	Assume that $\alpha \in (\sqrt{2},2)$, $\sigma_1 > 0$ and $\sigma_2>0$. We let $r=\frac{1}{8\alpha^2}$, then there exists $C>0$ such that 
	\begin{align*}
		\sup_{t\leq T}
		\e[|X_t-X_t^{n^{r},n}|]
		\leq C(\log n)^{-1}.
	\end{align*}
	Assume that $\alpha \in (\sqrt{2},2)$, $\sigma_1 = 0$ and $\sigma_2 > 0$ then
		for any $p \in (0,\alpha^2-2)$, then there exists $C_{T,p}>0$ such that 
	\begin{align*}
		\sup_{t\leq T}
		\e[|X_t-X_t^{{n^\ell},n}|]
		\leq
		C_{T,p}\, n^{-L}
	\end{align*}
	where
	$\ell = \frac{(\frac{2}{\alpha}-1)\frac{1}{4\alpha}}{1+p(1-\frac{\alpha}{2})}$ and $L = \left(\frac{2}{\alpha} - 1\right)\frac{1}{4\alpha} - \ell$, and the constant $C_{T,p}$ convergences to infinity as $p\uparrow (\alpha^2-2)$.
\end{Cor}
\begin{proof}
	By applying Theorem \ref{main_1} and Theorem \ref{main_2} with $H=n^{r}$,
	\begin{align*}
	\sup_{t\leq T}
		\e[|X_t-X_t^{n^r,n}|]
		\leq C(r \log n)^{-1}
		+\{ (\log n)^{-1}+n^r(\log n)^{\alpha-1}n^{-\frac{1}{4\alpha^2}}\}
		\leq C(\log n)^{-1}.
	\end{align*}
In the case $\sigma_1=0$, we choose $H=n^\ell$ where $\ell = \frac{(\frac{2}{\alpha}-1)\frac{1}{4\alpha}}{1+p(1-\frac{\alpha}{2})}$, then
	\begin{align*}
	\sup_{t\leq T}
		\e[|X_t-X_t^{n^{\ell},n}|]
	& \leq CH n^{-(\frac{2}{\alpha}-1)\frac{1}{4\alpha}}
	+C_{T,p}H^{-p(1-\alpha/2)}
	\leq C_{T,p}\, n^{-L}
	\end{align*}
where $L = \left(\frac{2}{\alpha} - 1\right)\frac{1}{4\alpha} - \ell$
\end{proof}

\begin{remark}
It is slightly unpleasant to have restriction on the $\alpha$ parameter and, in the case where $\sigma_1= 0$, to present convergence rates which depend on the parameter $p\in (0,\alpha^2-2)$. However we want to point out that these are again technical issues. The restriction on $\alpha$ and the dependence of $p$ are due to estimate \eqref{EP_8_1}, where one is forced to use moment estimates in Lemma \ref{alpha_norm_0} as one has little information on the distribution of $X_t^H$ or $X_t$.
\end{remark}

\section{Numerical experiments}
Numerical experiments were performed to study the strong rate of convergence of the implicit approximation scheme $X^{H,n}$ give in \eqref{scheme}. Note that here we ignored the truncation constant $H$, since one is bounded, by design, by the largest number possible on a given computer architecture.  In the following, the terminal time is $T= 1$ and to illustrate the rate of convergence we compute
\begin{align*}
\mathbb{E}[|X_T^{H,2n} - X^{H,n}_T|] \approx \frac{1}{N} \sum_{i=1}^N |X_T^{H,2n} - X^{H,n}_T|
\end{align*}
using $N = 2^{20}$.  The implicit scheme $X^{H,n}$ is simulated using $n = 2^j$, for $j = 7,8,9,10$. The log-log plots of the error against grid size are given for some parameter values. 
\begin{figure}[h!]
  \begin{tabular}{@{}ccc@{}}
\includegraphics[scale = 0.4]{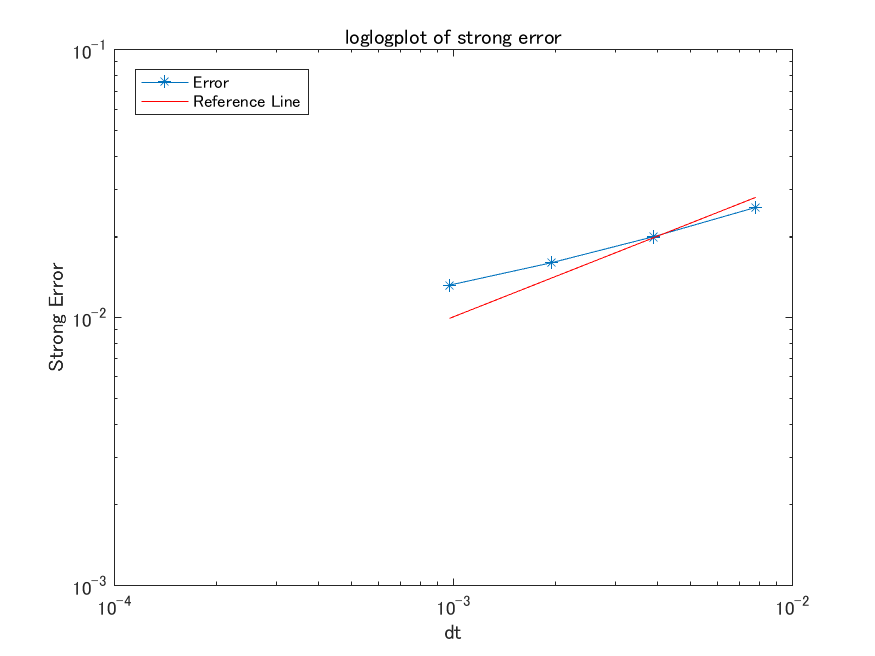}
\includegraphics[scale = 0.4]{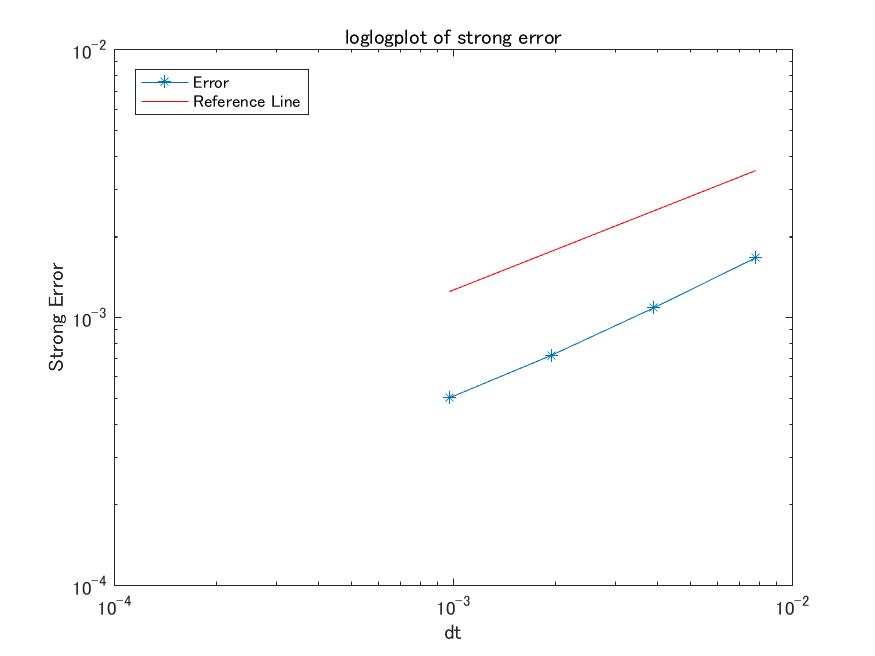}
  \end{tabular}
\caption{log-log plots illustrating the convergence of the scheme. The red line is the reference line with slop $1/2$. The graph on the left is generated with parameters $\sigma_1= 0.4$,   
$\sigma_2= 0.5$, $\alpha = 1.9$, $a = 1.03$, $k = 4$, $x_0 = 0.03$. The graph on the right is generated with parameters $\sigma_1= 0.4$,   
$\sigma_2= 0.5$, $\alpha = 1.005$, $a = 1.03$, $k = 4$, $x_0 = 0.03$.}
\end{figure}

We observe in the above that the rate of convergence appears lower for larger values of $\alpha$. This observation is intuitive, since the rate of convergence for the diffusion CIR, i.e. $\sigma_1 > 0$, $\sigma_2 =0$, is of order $1/2$, and on the other hand, Theorem \ref{main_1} tells us that in the jump case, i.e. $\sigma_1 = 0$, $\sigma_2 > 0$, the rate of convergence decreases as $\alpha \uparrow 2$. Therefore, when combined, the worst rate should prevail. Unfortunately, due to the limitation in techniques, we are unable to identify this theoretically when $\sigma_1>0$ and $\sigma_2>0$, and this will be the focus of future works. For experimental purposes, in the following pages, we present some graphs on the strong rate of convergence for different parameter regimes. 

\section*{Conclusion}
In this work, we devised a positivity preserving numerical scheme for the alpha-CIR process. We show that the scheme convergences and numerical experiments indicated that the strong rate is likely to be of order $1/2$ when $\alpha$ is small and decreases as $\alpha \uparrow 2$. 


\vfill\break 
\subsection*{Numerical experiments - Graphs on the rate of convergence}
\begin{figure}[htb]	
\centering
\vskip10pt
 \textbf{The effect of increasing $\sigma_2$}\par\medskip
  \begin{tabular}{@{}ccc@{}}
    \includegraphics[width=.47\textwidth]{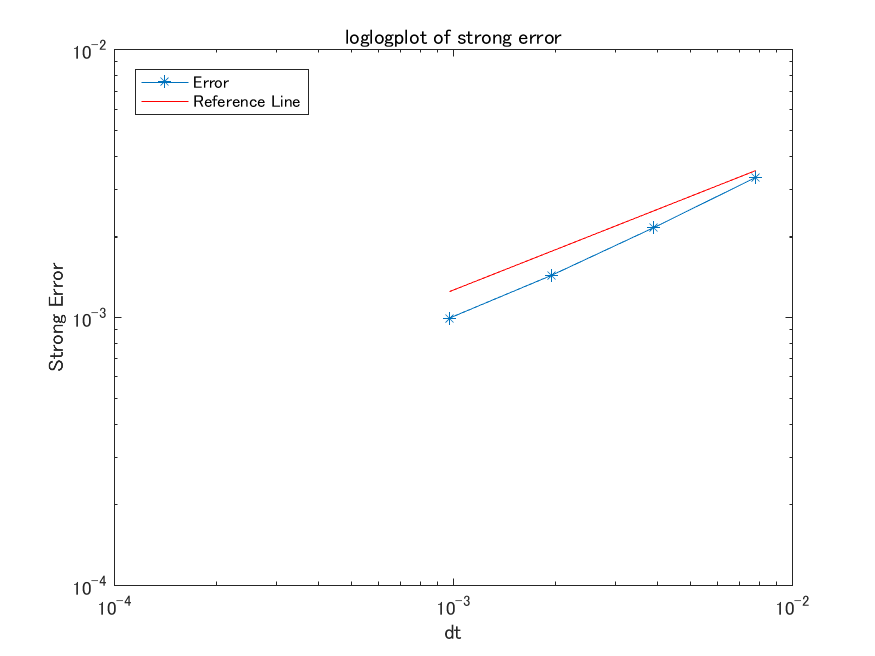} &
    \includegraphics[width=.47\textwidth]{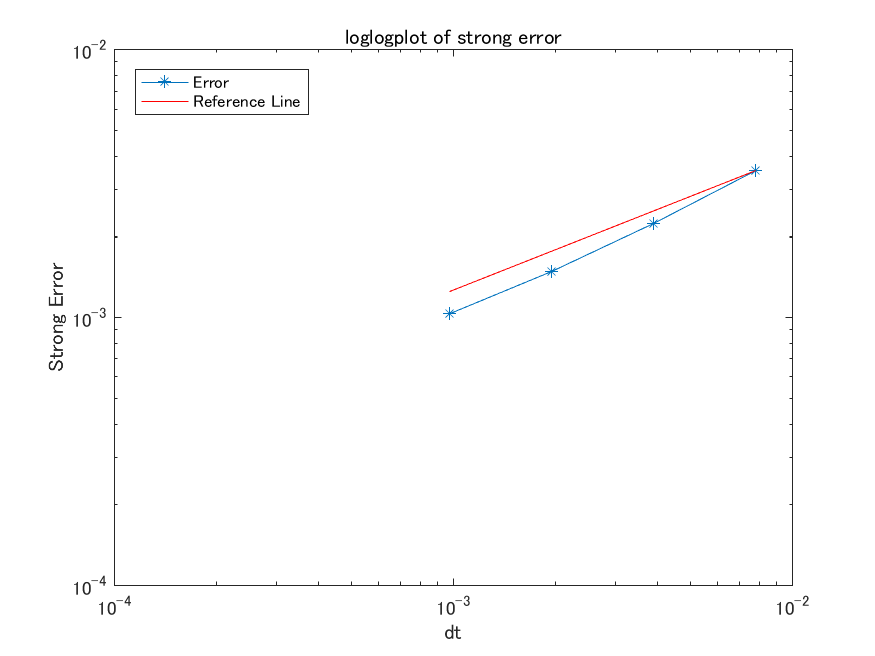} \\
    \includegraphics[width=.47\textwidth]{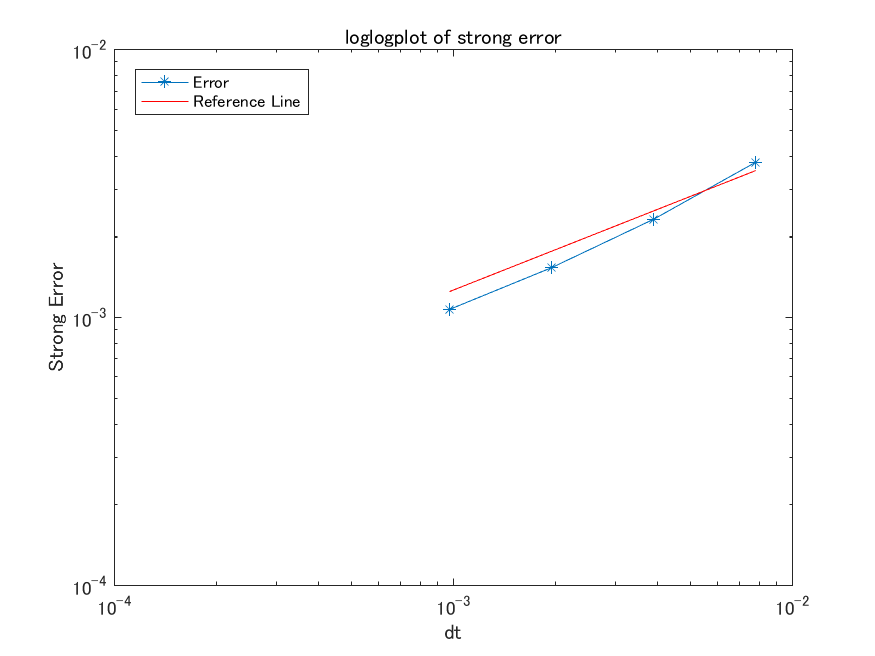} &
    \includegraphics[width=.47\textwidth]{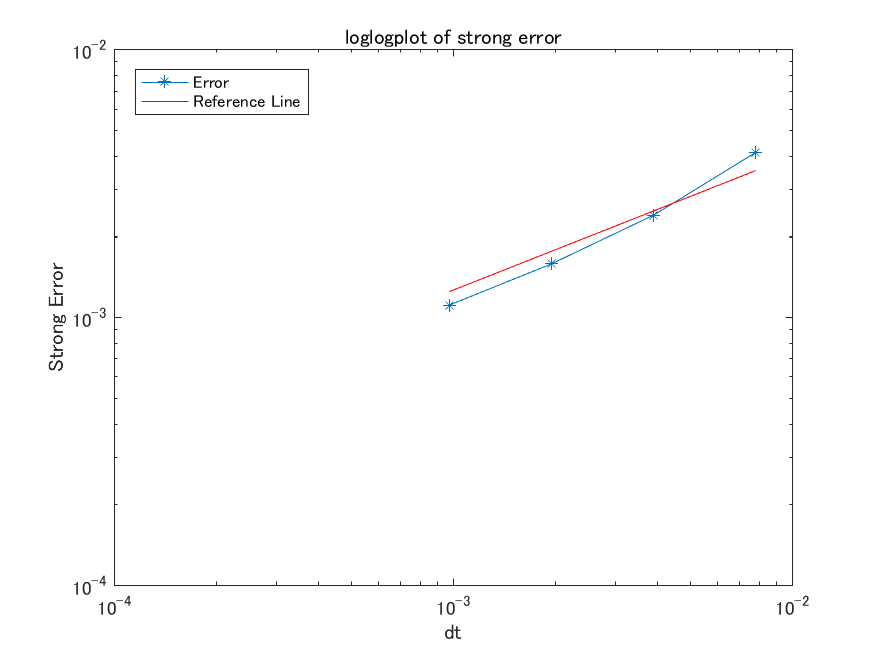} \\
    \includegraphics[width=.47\textwidth]{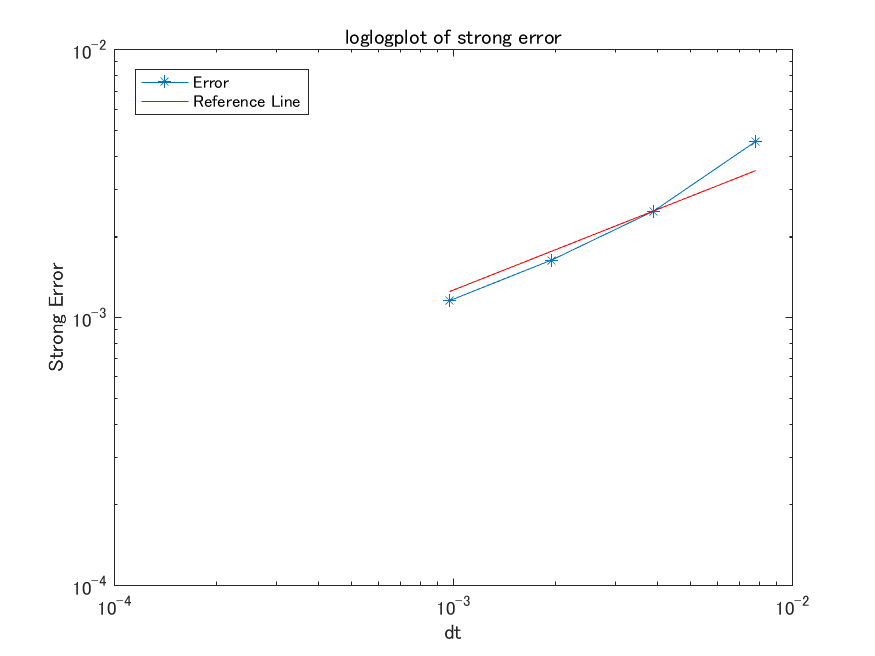} &
    \includegraphics[width=.47\textwidth]{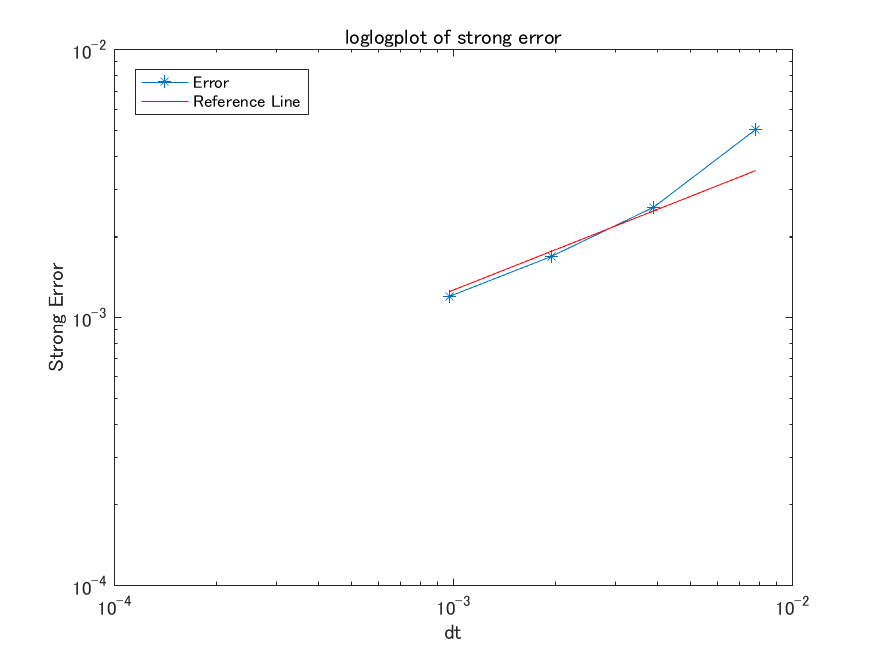}  
  \end{tabular}
\caption{log-log plots illustrating the rate of convergence of the implicit scheme. The red line is a reference line with slop $1/2$. The parameters are $\sigma_1= 0.5$,   
$\sigma_2= 0.05+0.05\times j$, for $j = 1,2,\dots, 6$, $\alpha = 1.05$, $a = 2$, $k = 3$, $x_0 = 0.03$. The graphs are to be read from left to right and top to bottom in increasing order of $j$. That is, the top left graph has $j = 1$ and the bottom right graph has $j = 6$.}
\end{figure}

\begin{Rem}
In the above, we investigate the effect of increases in $\sigma_2$. We observe that a increase in $\sigma_2$ increases the magnitude of the log-error. The numerical experiments suggests that the strong rate is $1/2$ when $\sigma_2$ and $\alpha$ are relatively small.
\end{Rem}

\vfill\break 

\begin{figure}[htb]
\centering
\vskip30pt
 \textbf{The effect of increasing $\alpha$ when $\sigma_2$ is small}\par\medskip
  \begin{tabular}{@{}ccc@{}}
    \includegraphics[width=.47\textwidth]{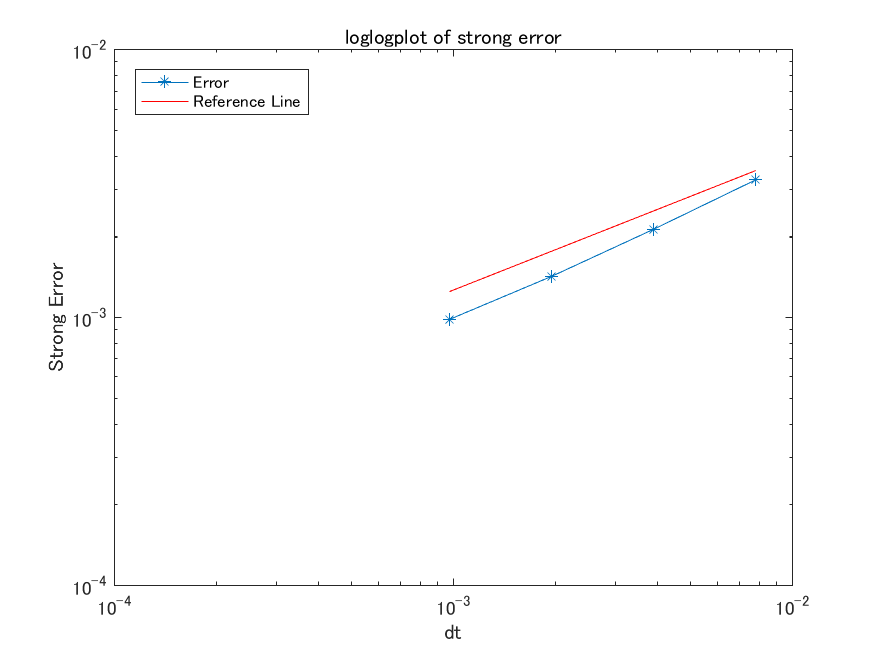} &
    \includegraphics[width=.47\textwidth]{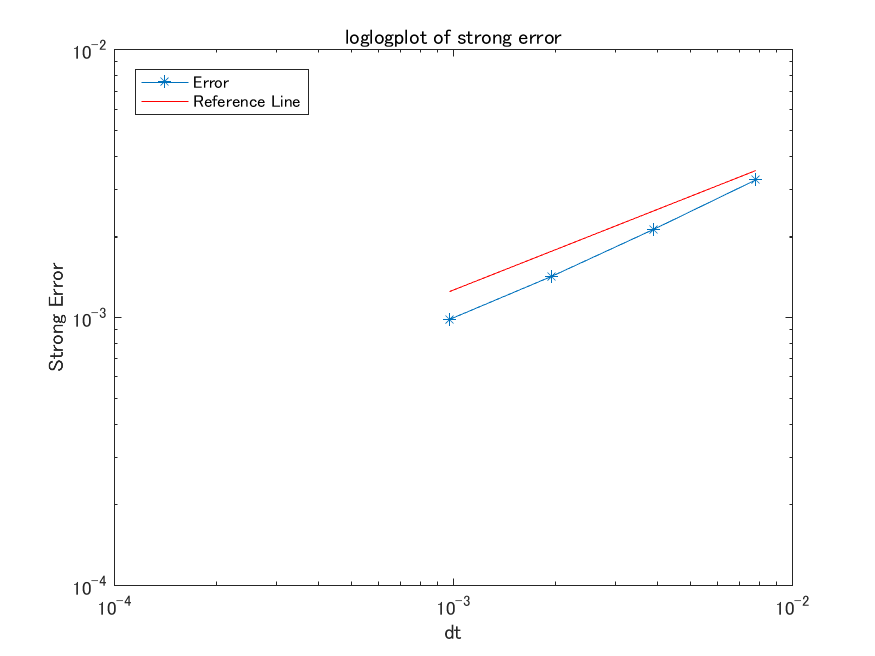} \\
    \includegraphics[width=.47\textwidth]{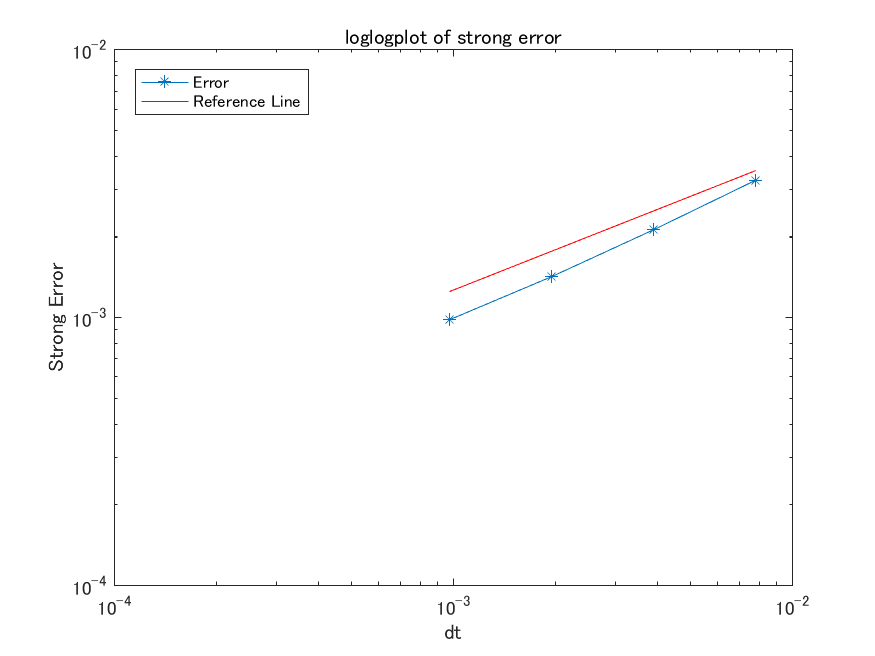} &
    \includegraphics[width=.47\textwidth]{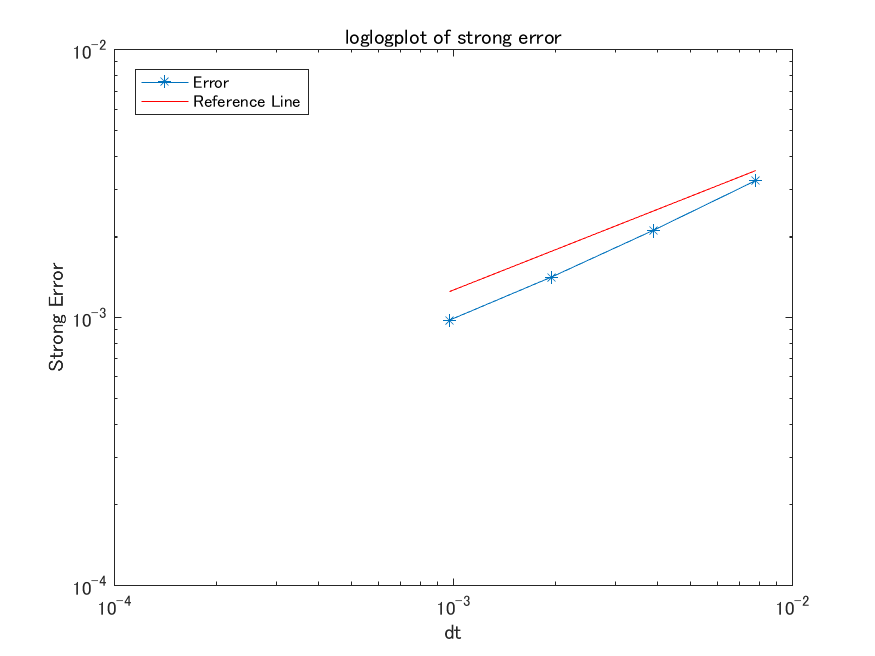} \\
    \includegraphics[width=.47\textwidth]{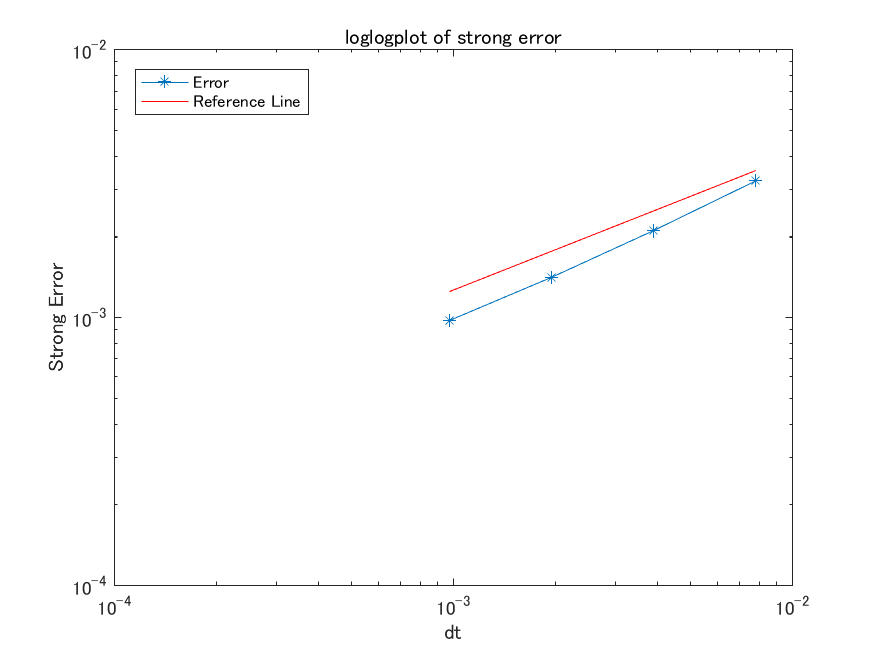} &
    \includegraphics[width=.47\textwidth]{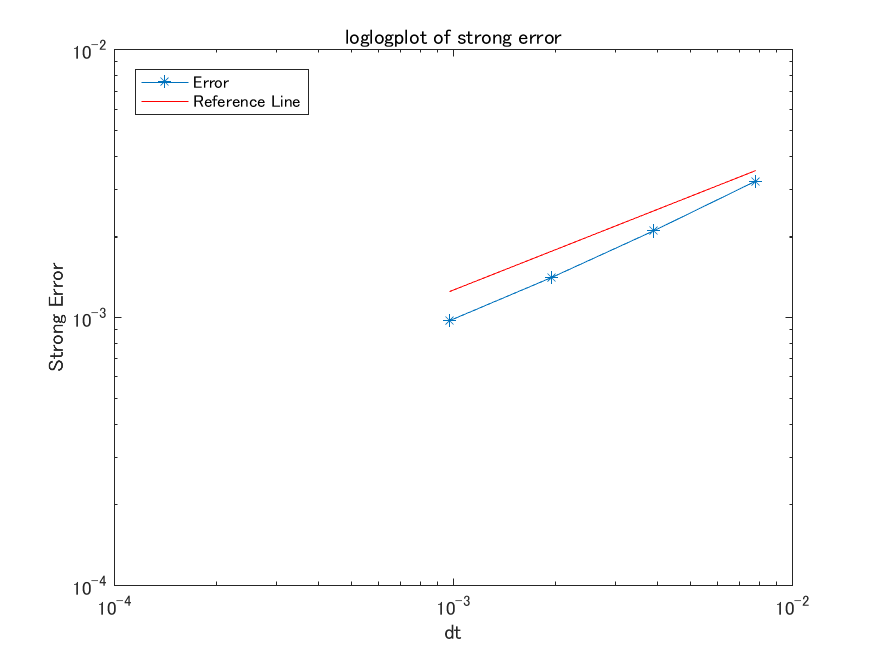} 
  \end{tabular}
\caption{log-log plots illustrating the convergence of the implicit scheme. The red line is a reference line with slop $1/2$. The parameters are $\sigma_1= 0.5$,  $\sigma_2= 0.05$, $\alpha = 1.10 + 0.05\times j$, for $j = 1,2,\dots, 6$, $a = 2$, $k = 3$, $x_0 = 0.03$. The graphs are to be read from left to right and top to bottom in increasing order of $j$. That is, the top left graph has $j = 1$ and the bottom right graph has $j = 6$.}
\end{figure}

\begin{Rem}
In the above graphs, we investigate the effect of increases in $\alpha$ in the case where $\sigma_2$ is small. We observe that when $\sigma_2$ is small then the increase in $\alpha$ (when $\alpha$ is relative small) have little effect to no effect on the rate of convergence or the magnitude of the log-error.
\end{Rem}

\vfill\break 

\begin{figure}[htb]
\centering
\vskip20pt
 \textbf{The effect of increasing $\alpha$ when $\sigma_1$ and $\sigma_2$ are large and are of the same magnitude}\par\medskip
  \begin{tabular}{@{}ccc@{}}
    \includegraphics[width=.47\textwidth]{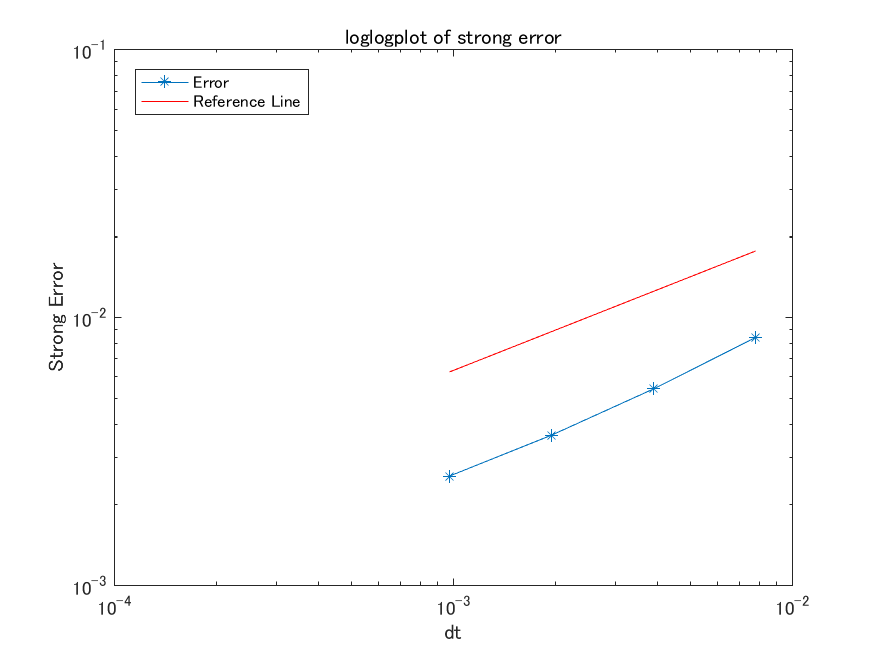} &
    \includegraphics[width=.47\textwidth]{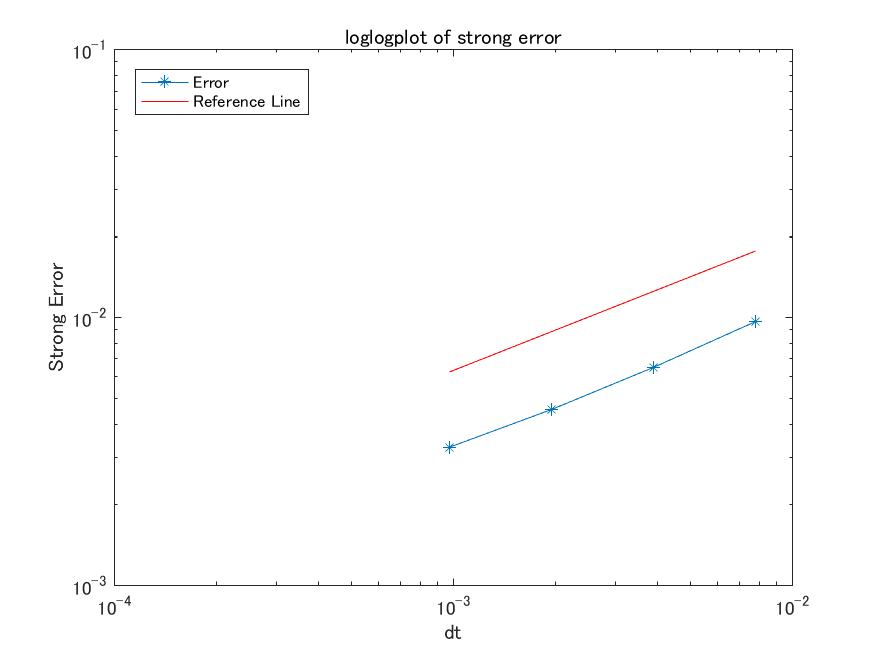} \\
    \includegraphics[width=.47\textwidth]{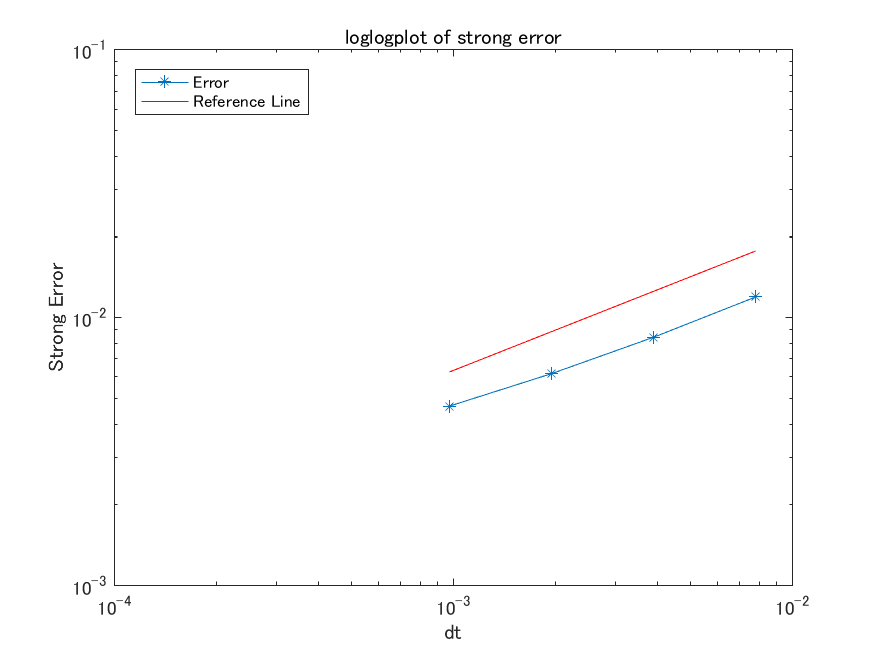} &
    \includegraphics[width=.47\textwidth]{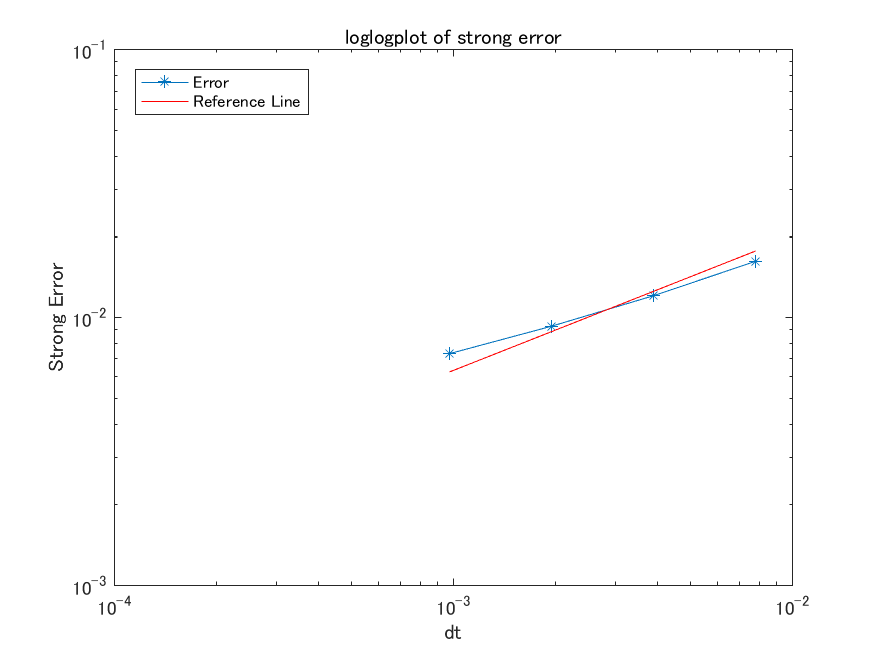} \\
    \includegraphics[width=.47\textwidth]{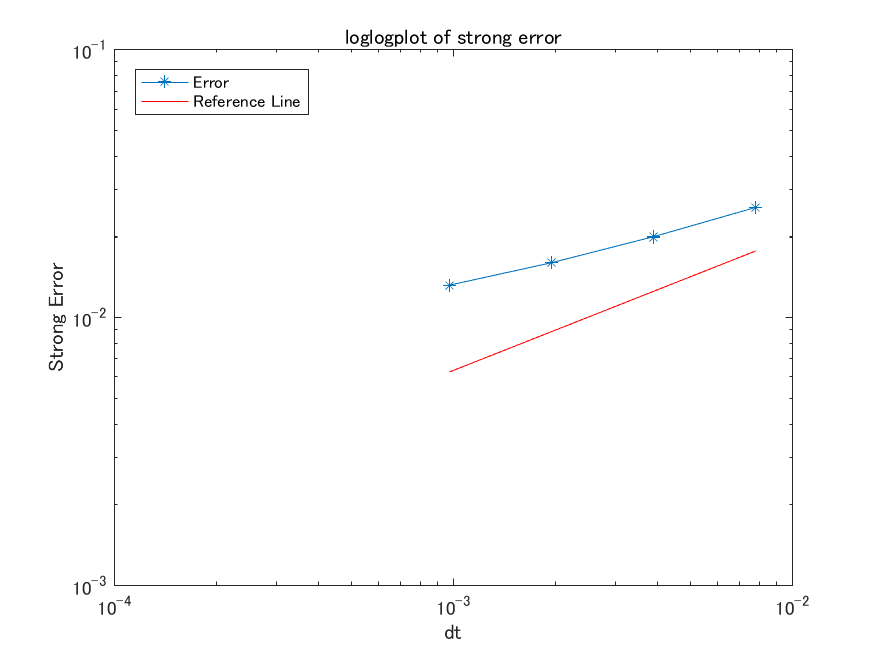} &
    \includegraphics[width=.47\textwidth]{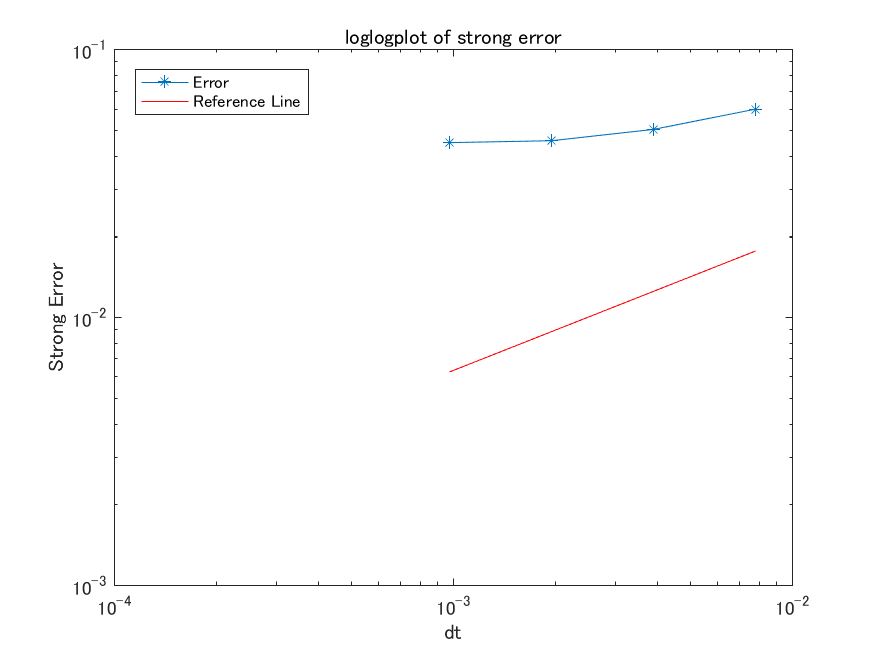} 
  \end{tabular}
\caption{log-log plots illustrating the convergence rate of the implicit scheme. The red line is a reference line with slop $1/2$. The parameters are $\sigma_1= 0.5$,  $\sigma_2= 0.5$, $\alpha = 1.3 + 0.1\times j$, for $j = 1,2,\dots, 6$, $a = 2$, $k = 3$, $x_0 = 0.03$. The graphs are to be read from left to right and top to bottom in increasing order of $j$. That is, the top left graph has $j = 1$ and the bottom right graph has $j = 6$.}
\label{f3.4}
\end{figure}

\begin{Rem}
In the above graphs, we investigate the effect of increases in $\alpha$ when $\sigma_1$ and $\sigma_2$ are of the same magnitude and are relatively large. We observe when $\alpha$ is close to two the rate of convergence decreases and the magnitude of the log-error also increases (note that the range of the $y$-axis is from $10^{-3}$ to $10^{-1}$).
\end{Rem}

\vfill\break

\begin{figure}[htb]
\centering
\vskip20pt
 \textbf{The effect of increasing $\alpha$ when $\sigma_2$ is large and $\sigma_1 = 0$}\par\medskip
  \begin{tabular}{@{}ccc@{}}
    \includegraphics[width=.47\textwidth]{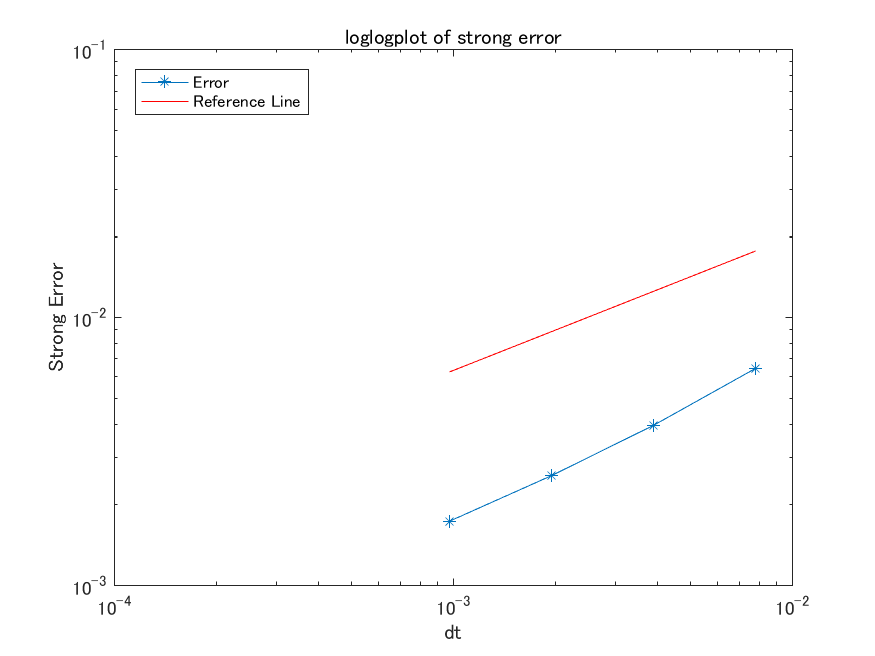} &
    \includegraphics[width=.47\textwidth]{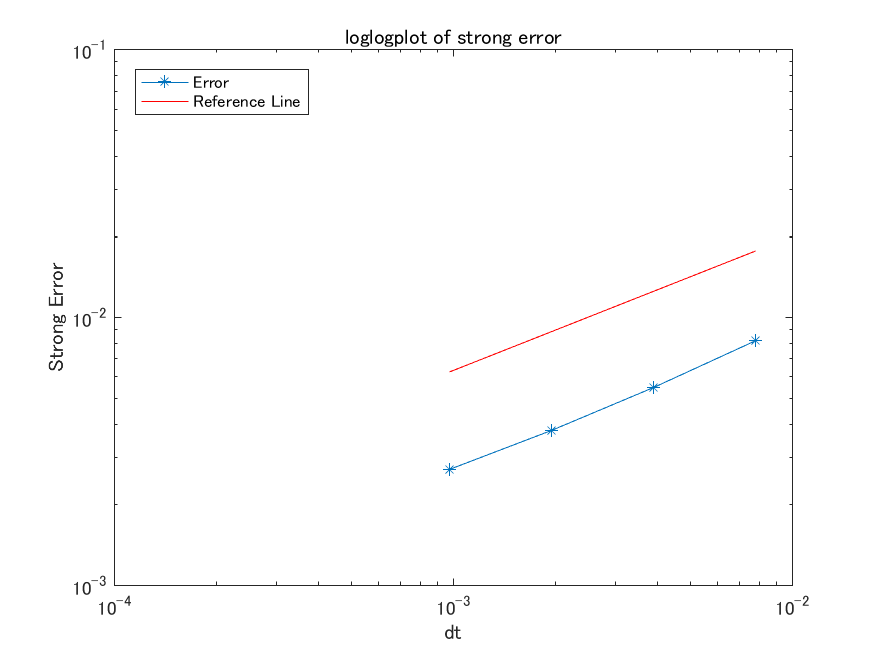} \\
    \includegraphics[width=.47\textwidth]{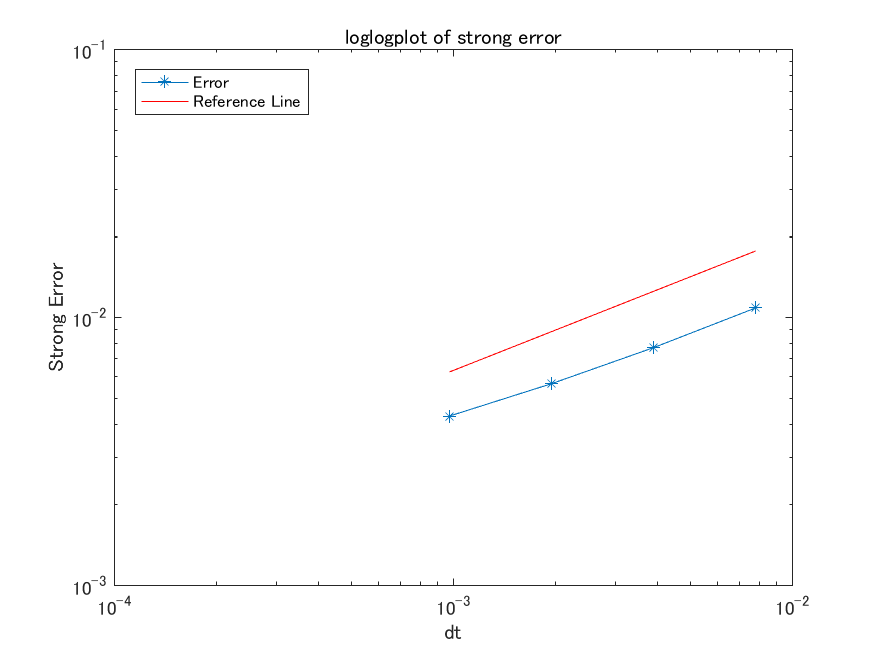} &
    \includegraphics[width=.47\textwidth]{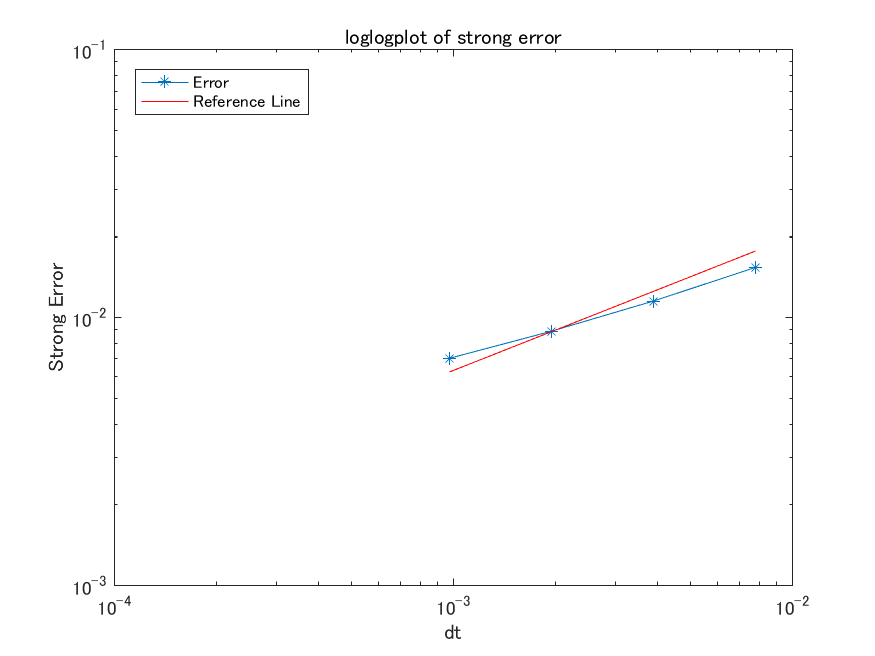} \\
    \includegraphics[width=.47\textwidth]{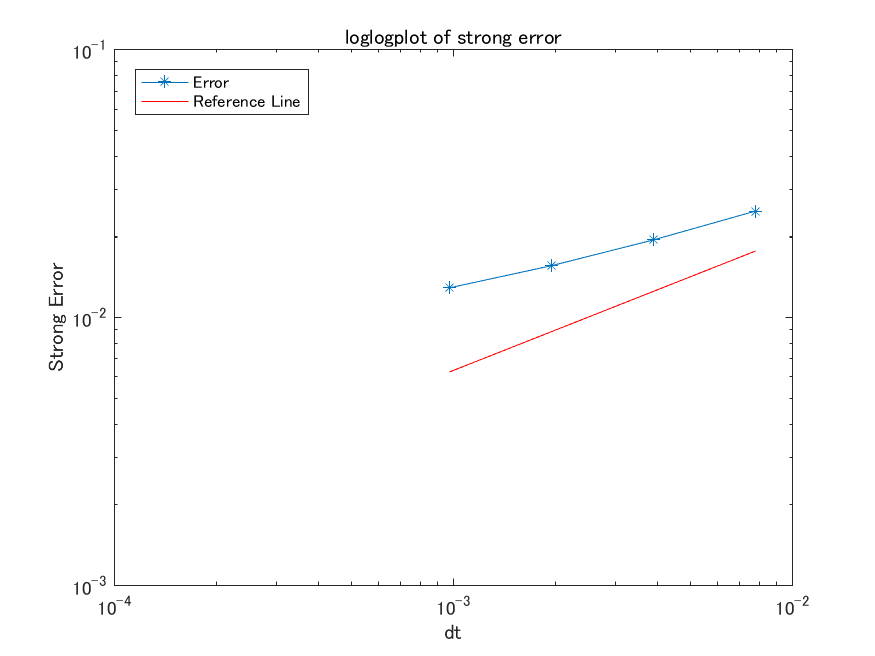} &
    \includegraphics[width=.47\textwidth]{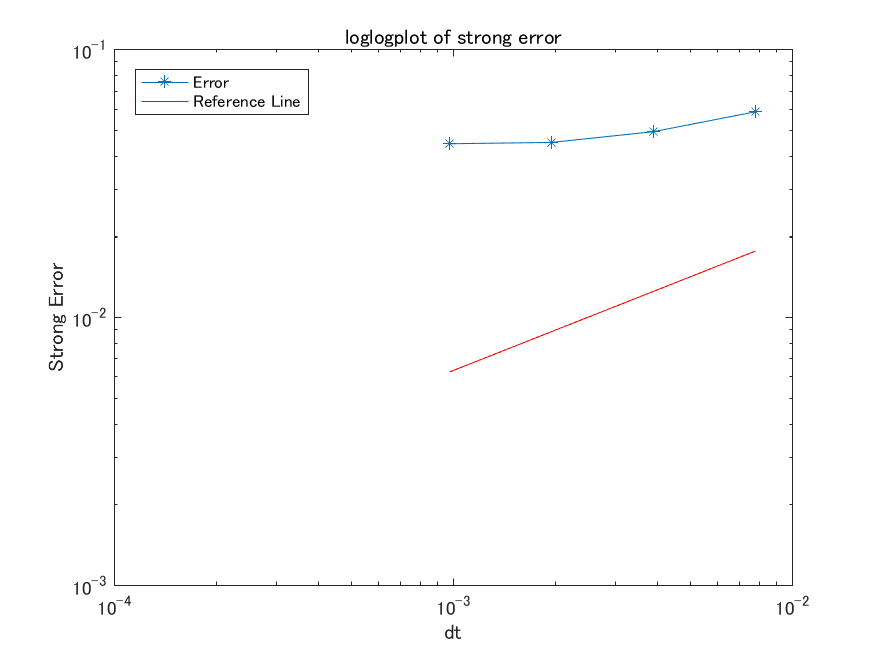} 
  \end{tabular}
\caption{log-log plots illustrating the convergence of the implicit scheme to the ``true" solution. The red line is a reference line with slop $1/2$. The parameters are $\sigma_1= 0$,  $\sigma_2= 0.5$, $\alpha = 1.3 + 0.1\times j$, for $j = 1,2,\dots, 6$, $a = 2$, $k = 3$, $x_0 = 0.03$. The graphs are to be read from left to right and top to bottom in increasing order of $j$. That is, the top left graph has $j = 1$ and the bottom right graph has $j = 6$.}
\end{figure}

\begin{Rem}
In the above graphs, we investigate the effect of increases in $\alpha$ when $\sigma_1=0$ and $\sigma_2$ is relatively large. The numerical results here are similar to those presented in Figure \ref{f3.4} where we observe that increases in $\alpha$ reduces the rate of convergence and the magnitude of the log-error also increases (note that the range of the $y$-axis is from $10^{-3}$ to $10^{-1}$). This is consistent with Theorem \ref{main_1}, where in the case $\sigma_1 = 0$, the strong rate of convergence decreases as $\alpha\uparrow 2$.
\end{Rem}

\section{Appendix}

\subsection{Moment estimate of $X$}\label{a6.1}

In this subsection, we show that for $h(x) = |x|^\frac{1}{\alpha}$ the solution of \eqref{SDE_0} has $\beta$-th moment for $\beta \in [1,\alpha)$.

\begin{Lem}\label{alpha_norm_0}
	For $\beta \in [1,\alpha)$, the $\beta$-th moment of $X$ is finite or more explicitly, there exists a constant $C_0>0$ such that
		\begin{align}
	\sup_{t\leq T} \e\big[|X_{t}|^{\beta}\big]
	\leq C_0(\alpha-\beta)^{-1} e^{C_0(\alpha-\beta)^{-1}T}.\label{l4.1const}
	\end{align}
\end{Lem}
\begin{proof}
	In the following, let $(\tau_m)_{m\in \mathbb{N}^+}$ be a localizing sequence of stopping times so that when stopped at $\tau_m$, all local martingales are martingales. By apply the It\^o formula to $X^{\beta}$, we obtain
	\begin{align}\label{alpha_norm_1}
	(X_{t \wedge \tau_m})^{\beta}
	=x_0^{\beta}
	+M_{t \wedge \tau_m}
	+I_{t \wedge \tau_m}
	+J_{t \wedge \tau_m}
	+K_{t \wedge \tau_m},
	\end{align}
	where we have set
	\begin{align*}
	{M}_t
	:=& \sigma_1 \beta \int_{0}^{t} 
	(X_{s})^{\beta-1/2} dW_{s}
	+\int_{0}^{t} \int_{0}^{\infty}
	\left\{
	(X_{s-}+\sigma_2 h(X_{s-})z)^{\beta}-(X_{s-})^{\beta}
	\right\}
	{\widetilde N}(ds,dz),\\
	{I}_t
	:=&\beta\int_{0}^{t} (X_{s})^{\beta-1} \{a-kX_{s}\}ds,
	\quad
	{J}_t
	:= \sigma_1^2 \frac{\beta(\beta-1)}{2} \int_{0}^{t} (X_{s})^{\beta-1}ds,\\
	{K}_t
	:=&
	\int_{0}^{t} \int_{0}^{\infty}
	\Big\{
	(X_{s-}+\sigma_2 h(X_{s-})z)^{\beta}-(X_{s-})^{\beta}-\sigma_2 h(X_{s-})z (X_{s-})^{\beta-1}
	\Big\}
	\nu(dz)ds.
	\end{align*}
	The martingale term $M_{t\wedge \tau_m}$ can be removed after taking the expectation.	
	Next we consider $K_{t \wedge \tau_m}$.
	For $z \in (0,1)$, by the second order Taylor's expansion for the map $x \mapsto x^{\beta}$, we have
	\begin{align*}
	(y+xz)^{\beta}-y^{\beta}-\beta xz y^{\beta-1}
	=\alpha (\beta-1) |xz|^2 \int_{0}^{1} \theta (y+\theta xz)^{\beta-2} d\theta
	\leq \frac{\beta (\beta-1) |xz|^2 }{y^{2-\beta}}.
	\end{align*}
	For $z \in [1,\infty)$, 
	by the first order Taylor's expansion for the map $x \mapsto x^{\beta}$ and the H\"older continuity of the map $x \mapsto x^{\beta-1}$, we have
	\begin{align*}
	\left| (y+xz)^{\beta}-y^{\beta}-\beta xz y^{\beta-1}\right|
	\leq \alpha |xz| \int_{0}^{1} \left| (y+\theta xz)^{\beta-1}-y^{\beta-1} \right| d\theta
	\leq \alpha |xz|^{\beta}.
	\end{align*}
	Hence, the expectation of $|K_{t \wedge \tau_m}|$ is bounded by
	\begin{align*}
	\beta(\beta-1)\e\big[
	\int_{0}^{t \wedge \tau_m} \int_{0}^{1}
	|X_s|^{\frac{2}{\beta}+\beta-2}
	z^2
	\nu(dz)ds
	\big]
	+\beta \e\big[
	\int_{0}^{t \wedge \tau_m} \int_{1}^{\infty}
	|X_s|
	z^{\beta}
	\nu(dz)
	ds
	\big].
	\end{align*}
	Since $\beta \in [1,\alpha)$, then $0< \frac{2}{\beta} + \beta -2 \leq \beta$ and there exists a constant $C>0$ such that $|x|^{\beta-1}\vee |x| \vee |x|^{\frac{2}{\beta}+\beta-2} \leq C(1+|x|^{\beta})$. Then by taking expectation on \eqref{alpha_norm_1}, and using the fact that $\nu(dz) \,\,\propto\,\, z^{-(1+\alpha)} dz$, we see that there exists $C_0>0$ such that
	\begin{align*}
	\e\big[|X_{t \wedge \tau_m}|^{\beta}\big]
	\leq C_0(\alpha-\beta)^{-1}
	+C_0 (\alpha-\beta)^{-1}\int_{0}^{t}
	\e\big[|X_{s \wedge \tau_m}|^{\beta}\big]
	ds.
	\end{align*}
	By Gronwall's inequality, we obtain
	\begin{align*}
	\e\big[|X_{t \wedge \tau_m}|^{\beta}\big]
	\leq C_0(\alpha-\beta)^{-1} e^{C_0(\alpha-\beta)^{-1}T}.
	\end{align*}
	Finally, we conclude by using Fatou's lemma.
\end{proof}

\subsection{Yamada-Watanabe Approximation Technique}\label{yamada}

We introduce below the Yamada and Watanabe approximation technique. For each $\delta \in (1,\infty)$ and $\varepsilon \in (0,1)$, we select a continuous function $\psi _{\delta, \varepsilon}: \real \to \real^+$ with support of $\psi _{\delta, \varepsilon}$ belongs to $[\varepsilon/\delta, \varepsilon]$ and is such that
\begin{align*} 
\int_{\varepsilon/\delta}^{\varepsilon} \psi _{\delta, \varepsilon}(z) dz
= 1 \quad \text{ and } \quad  0 \leq \psi _{\delta, \varepsilon}(z) \leq \frac{2}{z \log \delta}, \:\:\:z > 0.
\end{align*}
We define a function $\phi_{\delta, \varepsilon} \in C^2(\real;\real)$ by setting
\begin{align}
\phi_{\delta, \varepsilon}(x)&:=\int_0^{|x|}\int_0^y \psi _{\delta, \varepsilon}(z)dzdy.\label{yamada1}
\end{align}
It is straight forward to verify that $\phi_{\delta, \varepsilon}$ has the following useful properties: 
\begin{align*} 
&|x| \leq \varepsilon + \phi_{\delta, \varepsilon}(x), \text{ for any $x \in \real $}, \\ 
&0 \leq |\phi'_{\delta, \varepsilon}(x)| \leq 1, \text{ for any $x \in \real$} , \\
&\phi'_{\delta, \varepsilon}(x) \geq 0, \text{ for } x\geq 0 \text{ and } \phi'_{\delta, \varepsilon}(x) < 0, \text{ for } x< 0, \\
&\phi''_{\delta, \varepsilon}(\pm|x|)=\psi_{\delta, \varepsilon}(|x|)
\leq \frac{2}{|x|\log \delta}{\bf 1}_{[\varepsilon/\delta, \varepsilon]}(|x|)
\leq \frac{2\delta }{\varepsilon \log \delta},
\text{ for any $x \in \real \setminus\{0\}$}. 
\end{align*}

\begin{Lem}[Lemma 1.3 in \cite{LT}]\label{key_lem_0}
	Suppose that the L\'evy measure $\nu$ satisfies $\int_0^{\infty} \{z \wedge z^2\} \nu(dz) < \infty$.
	Let $\varepsilon \in (0,1)$ and $\delta \in (1,\infty)$.
	Then for any $x \in \real$, $y \in \real \setminus\{0\}$ with $xy \geq 0$ and $u>0$, it holds that
	\begin{align*}
	&\int_{0}^{\infty}
	\{\phi_{\delta,\varepsilon}(y+xz)-\phi_{\delta,\varepsilon}(y)-xz\phi_{\delta,\varepsilon}'(y)\} \nu(dz)\notag\\
	&\leq 2 \cdot \1_{(0,\varepsilon]}(|y|)
	\left\{
	\frac{|x|^2}{\log \delta} \left(\frac{1}{|y|} \wedge \frac{\delta}{\varepsilon}\right) \int_{0}^{u} z^2 \nu(dz)
	+ |x| \int_{u}^{\infty} z \nu(dz)
	\right\}.
	\end{align*}
\end{Lem}

\begin{Lem}[Lemma 1.4 in \cite{LT}]\label{key_lem12}
	Suppose that the L\'evy measure $\nu$ satisfies $\int_0^{\infty} \{z \wedge z^2\} \nu(dz) < \infty$.
	Let $\varepsilon \in (0,1)$ and $\delta \in (1,\infty)$.
	Then for any $x,x' \in \real$, $y \in \real$ and $u \in (0,\infty]$, it holds that
	\begin{align}\label{key_lem_122}
	&\int_{0}^{\infty}\left|
		\phi_{\delta,\varepsilon}(y+xz)-\phi_{\delta,\varepsilon}(y+x'z)-(x-x')z\phi_{\delta,\varepsilon}'(y)
	\right| \nu(dz) \notag\\
	&\leq 2	\left\{ \frac{\delta ( |x-x'|^2+|x'||x-x'|)}{\varepsilon\log \delta}\int_{0}^{u} z^2 \nu(dz)
	+ |x-x'| \int_{u}^{\infty} z \nu(dz)
	\right\}.
	\end{align}
	In particular, if $x'=0$, then
	\begin{align}\label{key_lem_123}
	& \int_{0}^{\infty}
	\{\phi_{\delta,\varepsilon}(y+xz)-\phi_{\delta,\varepsilon}(y)-xz\phi_{\delta,\varepsilon}'(y)\} \nu(dz)\\
	& \leq 2	\left\{ \frac{\delta |x|^2}{\varepsilon\log \delta}\int_{0}^{u} z^2 \nu(dz)
	+ |x| \int_{u}^{\infty} z \nu(dz)
	\right\}.\notag
	\end{align}
\end{Lem}

\subsection{Estimates of the probability that $D$ is negative}\label{d}

\noindent In the case where $Z$ is an $\alpha$-stable compensated L\'{e}vy process with $\alpha \in (1,2)$, where for $q\geq 0$, the moment generating function is given by $$\mathbb{E}\left[e^{-qZ_t}\right] = \exp\left(\frac{tq^{\alpha}}{\sin(\pi(\alpha-1)/2)}\right),$$
see Jiao et al. \cite{JMS}. The support of $Z_t$ is not bounded below and it is not possible to find conditions on the parameters which guarantee that the process $D$ is non-negative. 

\begin{proof}[Proof of Lemma \ref{dnegative}]
One can estimate the conditional probability using the conditional Laplace transform, and  for any $m>0$,
\begin{align*}
& \mathbb{E}\big[e^{-mD_{t_{i+1}}}\big|\, \F_{t_i}\big] \\
&= \exp\left(-m\left(a - \sigma_1^2/2\right)\Delta t_i\right) 
\mathbb{E}\big[\exp\big(-m\big(x + \sigma_2 (|x|^\frac{1}{\alpha}\wedge H) \Delta Z_{t_i} \big) \big) \big] \Big|_{x= X^{H,n}_{t_i}}\\
	& =  \exp\left(-m\left(a - \sigma_1^2/2\right)\Delta t_i\right) \exp\left(-m X^{H,n}_{t_i} \right)\exp\left(\frac{m^\alpha \Delta t_i \sigma_2^\alpha (|X^{H,n}_{t_i}|\wedge H^\alpha)}{\sin(\pi(\alpha-1)/2)}\right)\\
	&\leq \exp\left(-m\left(a - \sigma_1^2/2\right)\Delta t_i\right) \exp\left(\left(\frac{m^{\alpha-1} \Delta t_i \sigma_2^\alpha}{\sin(\pi(\alpha-1)/2)} - 1 \right) m X^{H,n}_{t_i} \right).
\end{align*}
By setting $m$ as the solution to
$$\frac{m^{\alpha-1} \Delta t_i \sigma_2^\alpha}{\sin(\pi(\alpha-1)/2)} - 1 = 0$$ we eliminate $X^{H,n}_{t_i}$ from the above expression, giving the upper bound
$$
\exp\left(-\left(a-\sigma_1^2/2\right)\left(\sin\left(\pi(\alpha-1)/2\right)\right)^{\frac{1}{\alpha-1}}\sigma_2^{-\frac{\alpha}{\alpha-1}}\left(1/\Delta t_i\right)^{\frac{1}{\alpha-1}-1} \right).
$$
\end{proof}

\noindent {\bf Acknowledgement:} The authors wish to thank the anonymous referees for their careful readings and valuable advices on the writing of this article. The first author also wishes to thank Allan Loi for interesting discussions. The second author was supported by JSPS KAKENHI Grant Number 17H066833.

\end{document}